\documentclass{amsart}
\usepackage{enumerate}
\usepackage{tikz}
\usetikzlibrary{automata,positioning}
\usepackage{hyperref}
\usepackage{mathrsfs}
\usepackage{amsfonts}
\usepackage{latexsym}
\usepackage{epsfig}
\usepackage{indentfirst}
\graphicspath{{figures/}}
\usepackage{amsmath}
\usepackage{amsthm}
\usepackage{graphicx}
\usepackage{amssymb}
\usepackage{amscd}
\usepackage{latexsym}
\usepackage{color}
\DeclareGraphicsExtensions{.eps}
\input xy
\xyoption{all}

\newtheorem{thm}{Theorem}[section]
\newtheorem{lem}[thm]{Lemma}
\newtheorem{cor}[thm]{Corollary}
\newtheorem{prop}[thm]{Proposition}

\theoremstyle{definition}
\newtheorem{defn}[thm]{Definition}
\newtheorem{exam}[thm]{Example}
\newtheorem{conj}[thm]{Conjecture}
\newtheorem{que}[thm]{Question}

\theoremstyle{remark}
\newtheorem{rmk}[thm]{Remark}

\numberwithin{equation}{section}

\numberwithin{equation}{section} \makeatletter

\begin{document}

\title{Knot Floer homology of satellite knots with $(1,1)$-patterns}

\author{Wenzhao Chen}
\address{Max Planck Institute of Mathematics, Vivatsgasse 7, 53111 Bonn, Germany}
\email{chenwenz@msu.edu}
\thanks{}

\begin{abstract}
For pattern knots admitting genus-one bordered Heegaard diagrams, we show the knot Floer chain complexes of the corresponding satellite knots can be computed using immersed curves. This, in particular, gives a convenient way to compute the $\tau$-invariant. For patterns $P$ obtained from two-bridge links $b(p,q)$, we derive a formula for the $\tau$-invariant of $P(T_{2,3})$ and $P(-T_{2,3})$ in terms of $(p,q)$, and use this formula to study whether such patterns induce homomorphisms on the concordance group, providing a glimpse at a conjecture due to Hedden. 
\end{abstract}
\maketitle
\section{Introduction}
In 2001, Ozsv\'{a}th and Szab\'{o} defined a package of invariants for closed oriented 3-manifolds known as the \emph{Heegaard Floer homology} \cite{OS04}. Later on, a bordered theory for the hat-version Heegaard Floer homology is developed by Lipshitz-Ozsv\'{a}th-Thurston \cite{MR3827056}. This theory associates certain differential modules and A-infinity modules to 3-manifolds with parametrized boundaries, called type D ($\widehat{CFD}$) and type A structures ($\widehat{CFA}$) respectively. For a closed 3-manifold constructed from gluing two 3-manifolds with boundaries, the corresponding hat-version Heegaard Floer homology can be obtained by an appropriate tensor product of the type D structure of one piece and the type A structure of the other. Recently, Rasmussen-Hanselman-Watson gave a geometric interpretation of the bordered theory for 3-manifolds with torus boundary \cite{hanselman2016bordered}: the bordered invariants are interpreted as immersed curves decorated with local systems on $\partial M\backslash \{w\}$, where $w$ is a base point, and the pairing of type D and type A structures translates to taking Lagrangian intersection Floer homology of the curve-sets on the torus.  

The \emph{knot Floer homology} is a variant of the Heegaard Floer homology defined for null-homologous knots in closed oriented 3-manifolds, introduced by Ozsv\'{a}th-Szab\'{o} and independently by Rasmussen \cite{MR2065507, MR2704683}. One can regard this invariant as the filtered chain homotopy type of certain $\mathbb{Z}$-filtered chain complex $\widehat{CFK}(K)$ associated to a knot $K$. The bordered theory carries through this setting: it associates filtered bordered invariants to knots in 3-manifolds with parametrized boundaries, and the knot Floer homology for knots in 3-manifolds constructed by glueing can be obtained by tensoring the corresponding bordered invariants. This machinery is well suited for studying the knot Floer homology of \emph{satellite knots}. Recall given a pattern knot $P\subset S^1\times D^2$ and a companion knot $K\subset S^3$, the satellite knot $P(K)$ is constructed by gluing $(S^1\times D^2, P)$ to the knot complement $X_K=S^3-nb(K)$ of $K$ so that the meridians are identified, and the longitude of $S^1\times D^2$ is identified with the Seifert longitude of $K$. For example, the knot Floer homology of satellite knots obtained by cabling or applying the Mazur pattern were studied using this tool \cite{MR3217622,MR3589337,MR3134023}. 

As mentioned above, pairing unfiltered bordered invariants of 3-manifolds with torus boundary may be taken as Lagrangian intersection Floer homology of curves on the torus. In this paper, we seek to explore a counterpart for the pairing of the filtered type A structure $\widehat{CFA}(S^1\times D^2, P)$ and the (unfiltered) type D structure $\widehat{CFD}(X_K)$, and hence obtain the knot Floer homology of the satellite knot $P(K)$. 
\subsection{The main theorem}
Our main theorem will restrict to a class of pattern knots called \emph{$(1,1)$-pattern knots}. The aforementioned cabling and Mazur pattern belong to this class, as well as the Whitehead double operator. 

\begin{defn}
A pattern knot $P\subset S^1\times D^2$ is called a $(1,1)$-pattern knot if it admits a genus-one doubly-pointed bordered Heegaard diagram.
\end{defn}
We set up some conventions in order to state the main theorem. Whenever the context is clear, genus-one doubly-pointed bordered Heegaard diagrams will just be called by genus-one Heegaard diagrams. Let $(\Sigma, \alpha^a, \beta, w, z)$ be a genus-one Heegaard diagram for $P\subset S^1\times D^2$. Note we may view the objects $(\beta, w,z)$ as embedded in $\partial (S^1\times D^2)$, and the $\alpha^a$ arcs correspond to the meridian $\mu$ and longitude $\lambda$ of $\partial (S^1\times D^2)$. So the data contained in such a bordered Heegaard diagram can be equivalently be understood as a 5-tuple $(\beta, \mu,\lambda, w, z)\subset \partial(S^1\times D^2)$ (Figure \ref{figure, convert bordered diagram to 5 tuple}). We warn the reader that this set of data depends on the choice of Heegaard diagrams, and hence is not an invariant of the pattern knot. 
For a 3-manifold $M$ with torus boundary, we denote the immersed-curve invariant as $(\widehat{HF}(M),w)\subset \partial M$ and call it the Heegaard Floer homology of $M$. The main theorem is stated below. The readers who prefer a more visual presentation may first read Example \ref{example, illustrating the main theorem} and then come back to the following formal statement.

\begin{thm}\label{Pairing thm}
Given a $(1,1)$-pattern knot $P\subset S^1\times D^2$ and a companion knot $K$ in $S^3$. Let $(\widehat{HF}(X_K), w')\subset \partial X_K$ be the Heegaard Floer homology of knot complement $X_K$ of $K$, and let $(\beta, \mu,\lambda, w, z)\subset \partial (S^1\times D^2)$ be the 5-tuple corresponding to some genus-one Heegaard diagram for $P$.
Let $h: \partial X_K\rightarrow \partial (S^1\times D^2) $ be an orientation preserving homeomorphism such that 
\begin{enumerate}
\item[(1)] $h$ identifies the meridian and Seifert longitude of $K$ with $\mu$ and $\lambda$ respectively;
\item[(2)] $h(w')=w$;
\item[(3)] there is a regular neighborhood $U\subset \partial (S^1\times D^2)$ of $w$ such that $z \in U$, $U\cap (\lambda\cup \mu)=\emptyset$, and $U \cap h(\widehat{HF}(X_K))=\emptyset$.
\end{enumerate}
Let $\alpha=h(\widehat{HF}(X_K))$. 
Then there is a chain complex isomorphism $$\widehat{CFK}(\alpha,\beta, w, z) \cong \widehat{CFK}(S^3, P(K)).$$
Moreover, if $\alpha$ is connected, this isomorphism preserves the Maslov grading and Alexander filtration. 
\end{thm}

\begin{rmk}
When $\alpha(K)$ is not connected, the full grading information can still be recovered when provided extra data called ``phantom arrows". Roughly, these are arrows that connects different components of $\alpha(K)$ and does not alter the chain complex obtained by pairing. In general $\alpha(K)$ is not connected, yet it always has a distinguished component which is gives a nontrivial homology class in $H_1(\partial X_K)$, and when one is interested in computing the $\tau$-invariant of the satellite knot, it suffices to restrict to the distinguished component \cite{hanselman2018heegaard}. 
\end{rmk}
\begin{exam}\label{example, illustrating the main theorem}
In practice, Theorem \ref{Pairing thm} amounts to laying the bordered Heegaard diagram of the pattern knot over the imersed-curve diagram of the knot complement. We consider the Mazur pattern $M$ actting on the right-handed trefoil $T_{2,3}$. In Figure \ref{figure, pairing Mazur and Trefoil} (a), a 5-tuple corresponding the Mazur pattern is given on the left, the immersed curve for the trefoil complement is drawn on the right, and the pairing is given in the middle. By lifting the curves to the universal cover of the torus and doing isotopy, the pairing diagram can be presented as in Figure \ref{figure, pairing Mazur and Trefoil} (b). This is the minimal intersection diagram and hence the intersection points are in one-to-one correspondence with elements in $\widehat{HFK}(M(T_{2,3}))$. 
\begin{figure}[htb!]
\begin{center}
\includegraphics[scale=0.45]{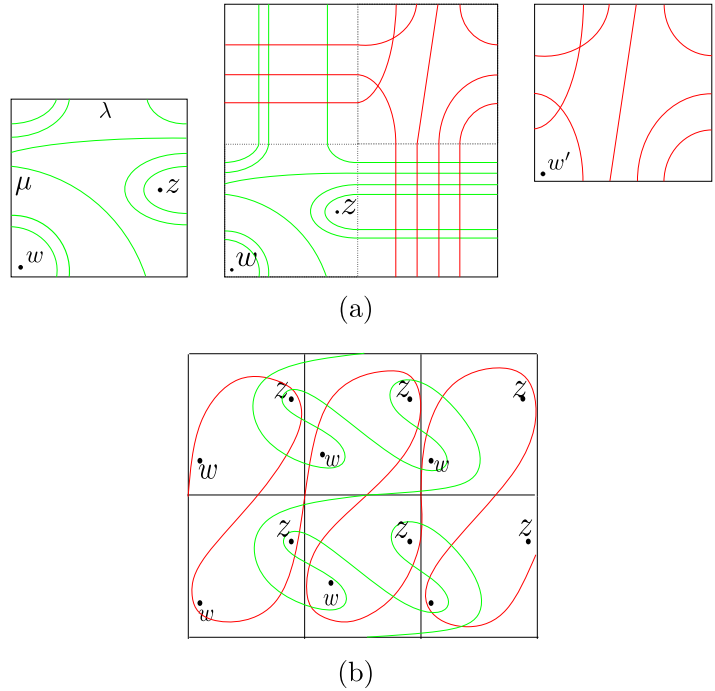}
\caption{The pairing diagram for $\widehat{CFK}(M(T_{2,3}))$}\label{figure, pairing Mazur and Trefoil}
\end{center}
\end{figure}

In addition to a quick computation of the rank of the knot Floer homology group, Theorem \ref{Pairing thm} also gives a handy way to compute the $\tau$-invariant of such satellite knots. In fact, one may repeatedly isotope $\beta(P)$ across the base point $z$ in the pairing diagram to eliminate intersection points with minimal Alexander filtration difference, and in the end only one intersection point is left, whose Alexander grading is exactly the $\tau$-invariant of the satellite knot (Figure \ref{figure, tau of Mazur of Trefoil}). This process is described in detail in Section \ref{section,algorithm for computing tau}.  
 \end{exam}
 
We point out the idea for proving Theorem \ref{Pairing thm} using Figure \ref{figure, idea of proof}. We work with the universal cover $\mathbb{C}$ of the torus. Given any embedded Whitney disk, we can push and collapse it to get a disk in a covering space of the bordered Heegaard diagram $(\Sigma, \alpha^a, \beta, w, z)$. This latter embedded disk gives rise to a type A operation in $\widehat{CFA}(P,S^1\times D^2)$, whose input of the elements in the torus algebra matches the those coming from the arcs on the $\alpha$ curve, i.e.\ type D operations in $\widehat{CFD}(X_K)$. This shows the correspondence of the differentials in the Lagrangian intersection Floer homology and in the box-tensor product. The detailed proof for Theorem \ref{Pairing thm} is in Section 3. 

\begin{figure}[htb!]
\begin{center}
\includegraphics[scale=0.55]{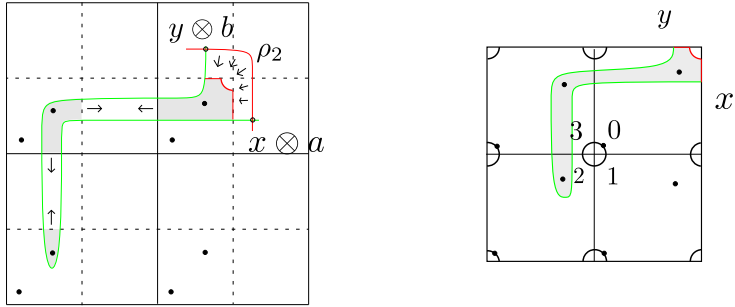}
\caption{A disk in the Lagrangian intersection pairing connecting $x\otimes a$ to $y\otimes b$ (left) can be pushed and collapsed (as indicated by the arrows) to give a disk in the bordered Heegaard diagram connecting $x$ to $y$ with compatible Reeb chords (right).}\label{figure, idea of proof}
\end{center}
\end{figure}
\subsection{Applications}
In this paper we apply Theorem \ref{Pairing thm} to study a question in knot concordance due to Hedden.
\begin{conj}[\cite{hedden2018satellites}]
The only homomorphisms on the knot concordance groups induced by satellite operators are the zero map, the identity, and the involution coming from orientation reversal.
\end{conj}
In this paper we only consider Hedden's conjecture in the smooth category, but we remark it is open in both the topological and smooth category.

Note if a pattern $P$ induces a group homomorphism, then it must be a slice pattern (i.e.\ $P(U)$ is a slice knot, where $U$ is the unknot). In this paper, we restrict our attention to unknot patterns (i.e.\ $P(U)=U$) that admit genus-one doubly-pointed Heegaard diagram. Note such patterns cannot be dealt with using the obrstruction coming from the Casson-Gordon invariant due to Miller \cite{miller2019homomorphism}.

Our first step classifies all $(1,1)$-unknot patterns using Theorem \ref{Pairing thm} and the fact that knot Floer homology detects the unknot \cite{MR2023281}. It turns out such patterns are identified with a simplest family of unknot patterns, which correspond to two-bridge links: remove a regular neighbourhood of one component of a two-bridge link leaves the other component as a pattern knot in the solid torus. 
\begin{thm}\label{Theorem, (1,1)-unknot pattern diagram gives two bridge pattern, introduction version}
Unknot patterns admitting genus-one doubly-pointed bordered Heegaard diagrams are in one-to-one correspondence with patterns determined by two-bridge links.  
\end{thm}
We actually prove stronger results which imply Theorem \ref{Theorem, (1,1)-unknot pattern diagram gives two bridge pattern, introduction version}: in Theorem \ref{classification of unknot pattern Heegaard diagrams} we classify all the genus-one Heegaard diagrams that give rise to unknot patterns, and in Theorem \ref{Theorem, (1,1)-unknot pattern diagram gives two bridge pattern} we give precise correspondence between such Heegaard diagrams and patterns determined by two-bridge links. 

Second, we give a formula for $\tau(P(T_{2,3}))$ and $\tau(P(-T_{2,3}))$, where $P$ is any pattern determined by a two-bridge link. Recall every two-bridge link admits a Schubert normal form parametrized by a pair of coprime integers $(p,q)$ such that $p$ is even; denote such a link by $b(p,q)$.    

\begin{thm}\label{Theorem, tau invariant formula}
Let $P$ be a pattern knot obtained from a two-bridge link $b(p,q)$ such that $q>0$. Let $w(p,q)$ be the winding number of $P$, and define $\sigma(a,b)=\sum_{i=1}^{a-1}(-1)^{\lfloor \frac{ib}{a} \rfloor}$. Then 
 $$\tau(P(T_{2,3}))=\max(\frac{|w(p,q)|+\sigma(\frac{p}{2}-q,q)}{2}+1,0),$$
 and
 $$\tau(P(-T_{2,3}))=\min(\frac{-|w(p,q)|+\sigma(\frac{p}{2}-q,q)}{2},0).$$
\end{thm}
\begin{rmk}
Theorem \ref{Theorem, tau invariant formula} is of independent interest. When restricting the companion knot to the trefoil or the left-handed trefoil, Theorem \ref{Theorem, tau invariant formula} unifies previous results on $\tau$-invariant of satellite knots: the $(n,1)$-cable corresponds to $b(2n,1)$, the Whitehead double corresponds to $b(8,3)$, and the Mazur pattern corresponds to $b(14,5)$ \cite{MR2372849}\cite{MR2511910}\cite{MR3217622}\cite{MR3589337}. In fact, when fixed to an aforementioned specific pattern knot, one can reprove the satellite formula using the technique in this paper.  
\end{rmk}
The assumption that a pattern $P$ induces a homomorphism on the concordance group constrains the behavior of the $\tau$-invariant under the action by $P$, i.e. $\tau(P(K))=|w(P)|\tau(K)$ for any knot $K$ (see the proof of Corollary \ref{Corollary, when two-bridge patterns do not homomorphism} in Section 6 or Proposition 5.4 of \cite{miller2019homomorphism}). This together with Theorem \ref{Theorem, tau invariant formula} implies

\begin{cor}\label{Corollary, when two-bridge patterns do not homomorphism}
Let $P$ be a pattern knot obtained from a two-bridge link $b(p,q)$ such that $q>0$. If $|w(p,q)|\neq 1$ or $\sigma(p,q)\neq -1$, then $P$ does not induce a group homorphism on the smooth knot concordance group. 
\end{cor}

Note that Hedden's conjecture in particular implies any pattern with winding number (modulo sign) greater than or equal to two does not induce homomorphism. Corollary \ref{Corollary, when two-bridge patterns do not homomorphism} confirms this within patterns obtained from two-bridge links. We actually wonder if the behavior of the $\tau$-invariant could completely answer this question. Namely,
\begin{que}
Is there a pattern $P$ with winding number $|w(P)|\geq 2$ such that $\tau(P(K))=|w(P)|\tau(K)$ for any knot $K$?
\end{que}
For slice patterns of winding number one, more subtle conditions are required to exclude the case in which the pattern is concordant (within the solid torus) to the core; the $\sigma$-function in Corollary \ref{Corollary, when two-bridge patterns do not homomorphism} is such a condition for patterns obtained from two-bridge links. Actually, examining examples with $|w|=1$ and $\sigma=-1$ leads to the following.
\begin{que}
Let $P$ be a pattern determined by a two-bridge link $b(p,q)$ with $q>0$. If $|w(P)|=1$ and $\sigma(\frac{p}{2}-q,q)=-1$, then is $P$ always concordant to the core of the solid torus? 
\end{que}
Finally, it might also be worth mentioning that it is unknown whether slice patterns of winding number one always act as the identity map in the topological category, and hence presumably it is hard to study this case by concordance invariants that are blind to the smooth-topological difference.   

\subsection{Immersed train tracks for general pattern knots}
It is natural to expect a similar immersed curve interpretation can be extended to include general pattern knots. At the course of writting this paper, it is not completely clear to the author how this can be achieved. Without genus-one bordered Heegaard diagrams, an algorithm must be given to translate the corresponding filtered bordered invariant to immersed train track on the torus. Such filtered invariant are often not reduced if we forget the filtration; this in particular prevents a direct application of the algorithm given in \cite{hanselman2016bordered}. 

One might possibly achieve this by using the bimodule point of view discussed in \cite{hanselman2019cabling}: $\widehat{CFDD}$-bimodules are expected to correspond to immersed surfaces (Lagrangians) in $T^2\times T^2$, and pairing with a $\widehat{CFA}$ whose immersed curve is $\gamma$ can be interpreted as intersecting the surface with $\gamma\times T^2$ and then projecting it to the second $T^2$. View a pattern knot equivalently as a two-component link. If one can successfully represent the $\widehat{CFDD}$-bimodule of this link complement as an immersed surface, then pairing the surface with the doubly-pointed Heegaard diagram for the pattern corresponding to the Hopf link would produce an immersed curve for the pattern knot.

In Section 7 we propose an approach in line with the spirit of working with filtered object: we introduce a notion called \emph{filtered extendability}, and give a way to produce immersed train tracks for filtered extended type D structures. Filtered extendability is an analogue of the extendability condition appeared in \cite{hanselman2016bordered}, and is automatically satisfied by all type D structures arised from pattern knots. In practice, the approach gets us the desired immersed curves and pairing. We hence speculate a favorable theory exists in general.

\subsection*{Organization} Section 2 contains preliminaries on bordered Heegaard Floer homology needed for the proof of Theorem \ref{Pairing thm}, which is given in Section 3. Section 4 contains a diagrammatic approach to compute the $\tau$-invariant. In Section 5 we prove Theorem \ref{Theorem, (1,1)-unknot pattern diagram gives two bridge pattern, introduction version}. In Section 6 we prove Theorem \ref{Theorem, tau invariant formula} and Corollary \ref{Corollary, when two-bridge patterns do not homomorphism}. Section 7 contains a brief discussion for immersed curves for general patterns. 
\subsection*{Acknowledgments} This project started when the author was a graduate student at Michigan State University and was partially supported by his advisor Matt Hedden's NSF grant DMS-1709016. The author would like to thank Matt Hedden for his enormous help. The author would also like to thank Abhishek Mallick for informing him of Ording's result used in this paper. The author is grateful to the Max Planck Institute for Mathematics in Bonn for its hospitality and support. 

\section{Preliminaries}
This section collects the relevant aspects of bordered Floer homology for 3-manifolds with torus boundary as a preparation for the proof of Theorem \ref{Pairing thm}. The experts may well skip this section. For the readers who would like to see more detail and other aspects of the theory, the author would recommend \cite{hanselman2016bordered,hanselman2018heegaard,MR3217622,MR3827056}.  

We outline the subsections: Section 2.1 recalls the definitions of a type D structure and a type A structure, the box tensor product, and how to define the filtered $\widehat{CFA}$ in terms of a genus-one doubly-pointed bordered Heegaard diagram; Section 2.2 explains how to interpret $\widehat{CFD}$ as train tracks on the torus; Section 2.3 explains gradings of the bordered Floer package.
\subsection{Bordered Floer homology}
We focus on bordered manifolds with torus boundary. To such manifolds, Lipshitz, Oszv\'ath, and Thurston associated a type D structure and a type A structure over the torus algebra up to certain quasi-ismorphism. 

The torus algebra $\mathcal{A}$ is given by the path algebra of the quiver shown in Figure \ref{figure, quiver for the torus algebra}. As a vector space over $\mathbb{F}$, $\mathcal{A}$ has a basis consisting of two idempotent elements $\iota_0$ and $\iota_1$, and six ``Reeb" elements: $\rho_1$, $\rho_2$, $\rho_3$, $\rho_{12}=\rho_1\rho_2$, $\rho_{23}=\rho_2\rho_3$, $\rho_{123}=\rho_1\rho_2\rho_3$.
\begin{figure}[ht!]
\center
\includegraphics[scale=0.5]{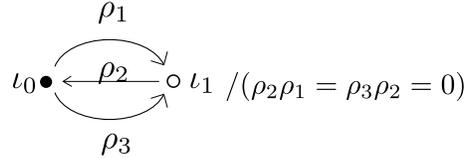}
\caption{The path algebra of this quiver with the specified relation is the torus algebra.}\label{figure, quiver for the torus algebra}
\end{figure}

Denote by $\mathcal{I}=\langle\iota_0\rangle \oplus \langle\iota_1\rangle\subset \mathcal{A}$ the ring of idempotents. A type D structure over $\mathcal{A}$ is a unital left $\mathcal{I}$-module $N$ equipped with an $\mathcal{I}$-linear map $\delta: N \rightarrow \mathcal{A}\otimes_{\mathcal{I}} N$ satisfying the compatibility condition $$(\mu\otimes \mathbb{I})\circ(\mathbb{I} \otimes \delta )\circ \delta=0.$$
Let $\delta_1=\delta$, and for $k=2,3,4,...$, inductively define maps $\delta_k=(\mathbb{I}^{\otimes {k-1}}\otimes \delta_1)\circ\delta_{k-1}$. A type D structure is bounded if $\delta_k=0$ for all sufficiently large $k$; in this paper we will always work with bounded type D structures. 

A type A structure is a right unital $\mathcal{I}$-module $M$ with a family of maps $m_{i+1}: M\otimes \mathcal{A}^{\otimes i}\rightarrow M$, $i\geq 0$ such that $$0=\sum_{i=1}^{n}m_{n-i}(m_i(x\otimes a_1\otimes\cdots\otimes a_{i-1})\otimes \cdots \otimes a_{n-1})+\sum_{i=1}^{n-2} m_{n-1} (x\otimes \cdots \otimes a_i a_{i+1} \otimes \cdots a_{n})$$ and $$m_2(x,1)=x$$ $$m_i(x,\cdots ,1, \cdots)=0,\ \ i>2$$

A type A structure $M$ and a type D structure $N$ can be paired up to give a chain complex via the \emph{box tensor product} $M\boxtimes N$. As a $\mathbb{F}$ vector space, $M\boxtimes N$ is isomorphic to $M\otimes_{\mathcal{I}} N$. The differential is defined by $$ \partial (x\otimes y)=\sum_{i=0}^{\infty}(m_{i+1} \otimes \mathbb{I}_{N})(x\otimes \delta_i(y))$$
Requiring the type D structure to be bounded guarantees the sum in the above equation is finite, and hence the box tensor product is well-defined.

Given two 3-manifolds $Y_1$ and $Y_2$ with torus boundary, let $h_1:T^2\rightarrow \partial(Y_1)$ and $h_2:-T^2\rightarrow \partial Y_2$ be diffeomorphisms parametrizing the boundaries, and let $Y=Y_1\cup _{h_1\circ h_2^{-1}}Y_2$ be the glued-up manifold. Lipshitz, Ozsv\'{a}th, and Thurston associated to $(Y_1, h_1)$ a type A structure $\widehat{CFA}(Y_1)$ and $(Y_2,h_2)$ a type D structure $\widehat{CFD}(Y_2)$, and showed that the box tensor product $\widehat{CFA}(Y_1)\boxtimes \widehat{CFD}(Y_2)$ is homotopy equivalent to $\widehat{CF}(Y)$. In the case when there is a knot $K\subset Y_1$ such that the induced knot in the glued-up manifold $Y$ is null-homologous, then one can associate to $K\subset Y_1$ a filtered type A structure $\widehat{CFA}(Y_1,K)$ so that the box tensor product $\widehat{CFA}(Y_1,K)\boxtimes \widehat{CFD}(Y_2)$ is homotopy equivalent to $\widehat{CFK}(Y,K)$.

All of the aforementioned objects are defined in terms bordered Heegaard diagrams and involve counting certain J-holomorphic curves. For our purpose, below we only recall the definition of $\widehat{CFA}(Y_1,K)$ when the bordered Heegaard diagram is of genus one. In this case, one could avoid (hide) the J-holomorphic curve theory. 

A genus-one doubly-pointed bordered Heegaard diagram $D$ is a 5-tuple $(\Sigma,\alpha^a,\beta,w,z)$ such that
\begin{itemize}
\item[$\bullet$]$\Sigma$ is a compact, oriented surface of genus one with a single boundary component.
\item[$\bullet$]$\alpha^a$ consists of a pair of arcs $(a_1^a,a_2^a)$ properly embedded in $\Sigma$, such that $a_1^a\cap a_2^a=\emptyset$ and the ends of $a_1^a$ and $a_2^a$ appears alternatively on $\partial \Sigma$ 
\item[$\bullet$]$\beta$ is an embedded closed curve in the interior $\Sigma$ such that $\Sigma\backslash \beta$ is connected, and intersects $\alpha^a$ transversely.
\item[$\bullet$]A basepoint $w$ on $\partial \Sigma \backslash \partial \alpha^a$, and a basepoint $z$ in $\text{Int}(\Sigma)\backslash (\alpha^a\cup\beta)$, where $\text{Int}(\Sigma)$ denotes the interior of $\Sigma$.
\item[$\bullet$]Label the arcs on $\partial \Sigma$, so that $(\partial \Sigma, \alpha^a, w)$ is as shown in Figure \ref{figure, pointed match circle}. We use the symbol $I\in\{12,23,123\}$ to denote the arc obtained by concatenation of the arcs labeled by $1$, $2$, $3$ accordingly. 
\end{itemize}
\begin{figure}[ht!]
\begin{center}
\includegraphics[scale=0.45]{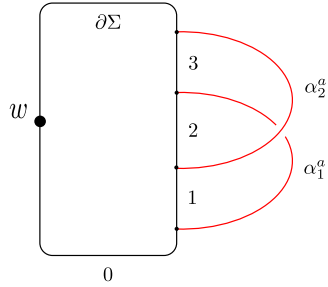}
\caption{Pointed match circle}\label{figure, pointed match circle}
\end{center}
\end{figure}
Such a diagram specifies a $3$-manifold with torus boundary and an oriented knot. The 3-manifold is obtained by attaching a 3-dimensional 2-handle to $\Sigma\times [0,1]$ along $\beta\times \{1\}$, and the knot is the union of two arcs on $\Sigma$: on connects $z$ to $w$ in the complement of $\beta$, the other connects $w$ to $z$ in the complement of $\alpha^a$. We do not explain, but simply point out the data also specifies a parametrization of the torus boundary. If a knot $K$ in a bordered $3$-manifold $Y_1$ can be represented by a genus-1 bordered Heegaard diagram, then we define a filtered type A module $\widehat{CFA}(Y_1,K)$ as following:
\begin{itemize}
\item[(1)]$\widehat{CFA}(Y_1,K)$ is generated by the set $\mathcal{G}=\{x|\ x\in \beta \cap \alpha^a\}$ as a $\mathbb{F}$ vector space.  
\item[(2)] For $x\in \mathcal{G}$, $\mathcal{I}$ acts on it as following
$$
x \cdot \iota_{0}=
\begin{cases}
x \ \text{if}\ x\in \alpha^a_1\cap \beta\\
0 \ \text{otherwise}
\end{cases}
$$
$$
x \cdot \iota_{1}=
\begin{cases}
x \ \text{if}\ x\in \alpha^a_2\cap \beta\\
0 \ \text{otherwise}
\end{cases}
$$
This induces a right $\mathcal{I}$-action on $\widehat{CFA}(Y_1,K)$.
\item[(3)]View $\Sigma=T^2-\text{Int}(B)$, where $B$ is a disk. Let $\tilde{\Sigma}$ be the covering space of $\Sigma$ which is obtained from the universal cover $\mathbb{C}$ of $T^2$ by removing the lifts of $B$. For the maps $m_{n+1}: M\otimes \mathcal{A}^{\otimes n}\rightarrow M$, $n\geq 0$, first define
$$m_{n+1}(x,\rho_{i_1},\cdots,\rho_{i_n})=\sum_{y\in \mathcal{G}} \#\mathcal{M}(x,y) y.$$
Here $i_j\in\{1,2,3,12,23,123\}$ for $j=1,\cdots,n$, and $\#\mathcal{M}(x,y)$ is the count (modulo 2) of index 1 embedded disks in $\tilde{\Sigma}$, such that when we traverse the boundary of such a disk with the induced orientation, we may start from a lift of $x$, traverse along an arc on (some lift of) $\alpha^a$, then along the arc ${i_1}$ on (some lift of) $\partial B$, $\cdots$, along the arc ${i_n}$, along an arc on $\alpha^a$ to a lift of $y$, finally it traverse along an arc on some lift of $\beta$ from $y$ to $x$.
Also define 
$$m_2(x,1)=x$$ $$m_i(x,\cdots ,1, \cdots)=0,\ \ i>2.$$
The above three equations determine $m_{n+1}$.
\item[(4)]Each term in $m_i$ has a relative Alexander grading difference that will specified in Subsection 2.2. 
\end{itemize}

\subsection{Gradings of the bordered Floer package}
The bordered Floer invariants are graded by certain (coset spaces of) non-commutative groups. When restricted to manifolds with torus boundary, the relevant grading group $G$ is defined to be
$$ G= \langle ( m;i,j) |\ m,i,j\in \dfrac {1}{2}\mathbb{Z} ,i+j\in \mathbb{Z} \rangle$$
with the group law
$$(m_1,i_1,j_1)\cdot (m_1,i_1,j_1)=(m_1+m_2+\dfrac{1}{2}\begin{vmatrix} i_{1} & j_{1} \\ i_{2} & j_{2} \end{vmatrix};i_1+i_2,j_1+j_2)$$ 
Here $m$ is called the \emph{Maslov component}, and $(i,j)$ is called the \emph{$\text{spin}^c$ component}.
In the presence of a knot we would also like to record the Alexander grading of the corresponding invariants, and this leads to using the enhanced grading group $\widetilde{G}=G\times \mathbb{Z}$. The new $\mathbb{Z}$ summand is called the \emph{Alexander factor}.

There are two elements in $\widetilde{G}$ will be relevant to us
$$\lambda=(1;0,0;0)\ \ \ \ \ \ \mu=(0;0,0;-1)$$

The torus algebra $\mathcal{A}$ is graded by $\widetilde{G}$ by setting 
$$
\begin{aligned}
&gr(\iota_i)=(0;0,0;0),\ \ \ \ \  i=1,2\\
&gr(\rho_1)=(-\dfrac{1}{2};\dfrac{1}{2},-\dfrac{1}{2};0)\\
&gr(\rho_2)=(-\dfrac{1}{2};\dfrac{1}{2},\dfrac{1}{2};0)\\
&gr(\rho_3)=(-\dfrac{1}{2};-\dfrac{1}{2},\dfrac{1}{2};0)
\end{aligned}
$$
and require $gr(\rho_I\cdot\rho_J)=gr(\rho_I)gr(\rho_J)$, for $I,J\in \{1,2,3,12,23,123\}$.

A type D structure $\widehat{CFD}(Y_2)$ decomposes as direct sum over $\text{spin}^c$-structures of $Y_2$. Fixing a $\text{spin}^c$ structure $\mathfrak{s}_2$, the corresponding component is graded by certain right coset space of $\widetilde{G}$
$$gr: \widehat{CFD}(Y_2,\mathfrak{s}_2)\rightarrow \widetilde{G}\slash \sim $$
The definition of this grading function involves utilizing some concrete Heegaard diagram and is not necessary for our purpose. Instead, we recall $gr$ satisfies $gr(\delta(x))=\lambda^{-1}gr(x)$ and $gr(\rho_I\otimes x)=gr(\rho_I)\cdot gr(x)$.

Similarly, a filtered type A structure decomposes as direct sum over $\text{spin}^c$-structures of $Y_1$. Fixing a $\text{spin}^c$ structure $\mathfrak{s}_1$, the corresponding component is graded by certain left coset space of $\widetilde{G}$
$$gr: \widehat{CFA}(Y_1,K, \mathfrak{s}_1) \rightarrow \sim\backslash\widetilde{G} $$
The property that will be relevant to us is if $B$ is a domain connecting $x$ to $y$, then
$gr(x)gr(B)=gr(y)$, where $$gr(B)=(-e(B)-n_x(B)-n_y(B); gr(\partial^\partial B); n_w(B)-n_z(B)).$$ Here $e(B)$ is the \emph{Euler measure}, $\partial^\partial B$ is the sequence $(\pm\rho_{i_1},\cdots \pm\rho_{i_k})$ of Reeb chords appearing at the boundary of $B$ (where the sign indicates the orientation), and $gr(\partial^\partial B)$ refers to the $\text{spin}^c$ component of $gr(\rho_{i_1})^{\pm 1}\cdots gr(\rho_{i_k})^{\pm 1}.$

The gradings on the type D and type A structure induces a grading on the box tensor product
$$
\begin{aligned}
gr:\widehat{CFA}(Y_1,K,\mathfrak{s}_1)\boxtimes \widehat{CFD}(Y_2,\mathfrak{s}_2)&\rightarrow \sim\backslash \widetilde{G} \slash \sim\\
x\otimes y&\mapsto gr(x)gr(y)
\end{aligned}
$$
Two elements $x_1\otimes y_1, x_2\otimes y_2 \in \widehat{CFA}(Y_1,K,\mathfrak{s}_1)\boxtimes \widehat{CFD}(Y_2,\mathfrak{s}_2)$ corresponds to the same $\text{spin}^c$ structure if and if there exist $M$ and $A$ such that $gr(x_1\otimes y_1)=gr(x_2\otimes y_2) \lambda^M \mu^A$. In the case, $M$ is the Maslov grading difference: $M(x_1\otimes y_1)-M(x_2\otimes y_2)$, and $A$ is the Alexander grading difference: $A(x_1\otimes y_1)-A(x_2\otimes y_2)$.

\subsection{Type D structure and immersed train tracks}
We recall how to represent a type D structure as immersed train tracks in a punctured torus.

A type D structure $N$ can be represented by an decorated graph. Let $N_i=\iota_i\cdot N$, $i=0,1$, then $N=N_0\oplus N_1$. Let $B_i$ be a basis for $N_i$, $i=1,2$. Fixing such a basis $B=B_0\cup B_1$ for $N$, we construct a decorated graph $\Gamma$. The vertices are in one-to-one correspondence with elements in $B$, and we label a vertex with $\bullet$ if it corresponds to an element in $B_0$, otherwise we label it with $\circ$. For two vertices corresponding to basis elements $x$ and $y$, we put a directed edge from $x$ to $y$ labeled with $I$ whenever $\rho_I\otimes y$ is a summand in $\delta(x)$, where $I\in \{\emptyset,1, 2, 3, 12, 23, 123\}$ .
We call a decorated graph is \emph{reduced} if none of the edges is labeled by $\emptyset$.

A reduced decorated graph can equivalently be represented by an immersed train track in the punctured torus $(T,w)$. More specifically, let $T=\mathbb{R}^2/\mathbb{Z}^2$, let $w=(1-\epsilon,1-\epsilon)$, and let $\alpha$ and $\beta$ be the image of the $y$ and $x$-axes. To construct the immersed train track, embed the vertices of the decorated graph into $T$ so that the $\bullet$ vertices lie on $\alpha$ in the interval $0\times [\frac{1}{4},\frac{3}{4}]$, and the $\circ$ vertices lie on $\beta$ in the interval $[\frac{1}{4},\frac{3}{4}]\times 0$, and then embed the edges in $T$ according to its label according to the rule as shown in Figure \ref{figure, train track assignment}.  We require the edges to intersect transversely. 
\begin{figure}[ht!]
\begin{center}
\includegraphics[scale=0.45]{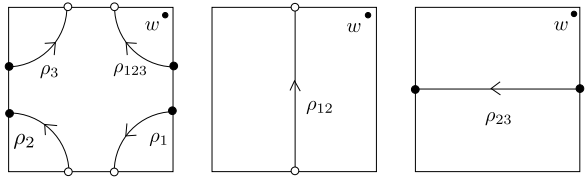}
\caption{Basic rules for representing the type D structures as train tracks.}\label{figure, train track assignment}
\end{center}
\end{figure}

In general, the train tracks thus obtained are usually not immersed curves. Work in \cite{hanselman2016bordered} shows that for type D structures arised from 3-manifolds with torus boundary, one can always pick some particularly nice basis, so that the train tracks obtained from the corresponding decorated graphs are immersed curves (possibly decorated by local systems).

\section{Proof of the main theorem}
\begin{proof}[Proof of Theorem \ref{Pairing thm}]
Let $\mu_K$ and $\lambda_K$ be the meridian and longitude of the companion knot $K$. Let $\alpha(K)$ be a train track representing $\widehat{CFD}(X_K,\mu,\lambda)$. For convenience, we assume $\alpha(K)$ comes from a restriction of $\widehat{HF}(X_K)$. ($\widehat{HF}(X_K)$ represents an exented type D structure, and $\alpha(K)$ is the curve-like sub-diagram of $\widehat{HF}(X_K)$ representing the underlying type D structure.) Let $H$=$(\Sigma, \alpha^a, \beta, w, z)$ be a genus-one bordered diagram for $(S^1\times D^2, P)$ corresponding to the standard meridian-longitude parametrization of the torus boundary. We first place $\alpha(K)$ and $(\beta,w,z)$ in a specific position on the torus $T^2$: Identify $T^2$ as the obvious quotient space of the squre $[0,1]\times [0,1]$ and divide the square into four quadrants by the segments $\{\frac{1}{2}\}\times [0,1]$ and $[0,1]\times \{\frac{1}{2} \}$. Include $\alpha(K)$ into the first quadrant and extend it horizontally/vertically. Include $H$ into the third quadrant so that $w$ is placed near $(0,0)$ and $\alpha^a$ is on the boundary of the third quadrant, then forget $\Sigma$ and $\alpha^a$, and extend $\beta$ horizontally/vertically. See Figure \ref{figure_pairing}. For the ease of notation, we set $\alpha=\alpha(K)$. 

\begin{figure}[ht!]
\center
\includegraphics[scale=0.5]{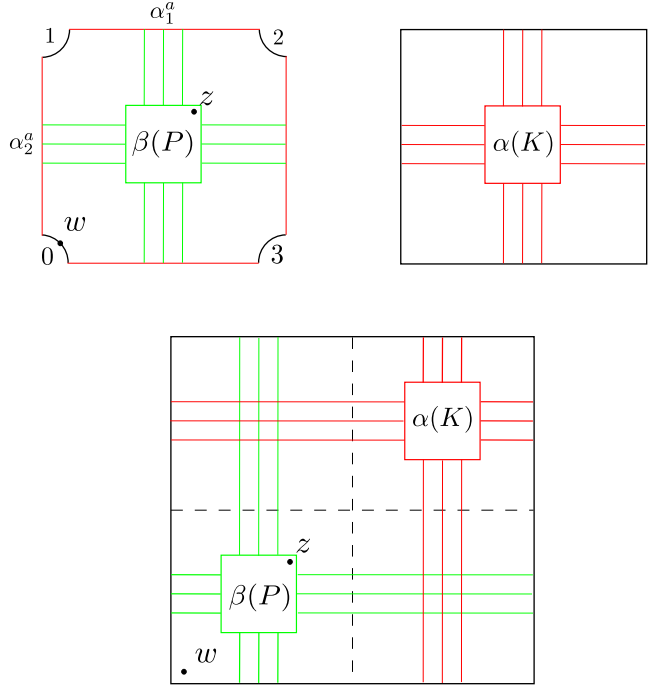}
\caption{Putting $\alpha(K)$ and $\beta(P)$ in a specific position.}\label{figure_pairing}
\end{figure}

We claim it suffices to prove 
\begin{equation}\label{Equation, box_tensor_product}
\widehat{CFK}(T^2, \alpha,\beta,w,z)\cong \widehat{CFA}(P,w,z)\boxtimes \widehat{CFD}(X_K,\mu_K,\lambda_K)
\end{equation}

To see the claim, first note the above placement of $\alpha$ and $\beta$ can be viewed as the result applying a specific representative of the homeomorphism $h$ as stated in the theorem. Secondly, regular homotopies of curve-like train tracks do not change the Lagrangian intersection Floer homology (Lemma 35 of \cite{hanselman2016bordered}).  

It is easy to see the isomorphism in \ref{Equation, box_tensor_product} as vector spaces: The intersections of $\alpha$ and $\beta$ only occur in the second and fourth quadrants, and those in the second quadrant are in one-to-one correspondence with the tensor product of the $\iota_0$ components of $\widehat{CFA}(P,w,z)$ and $\widehat{CFD}(X_K,\mu_K,\lambda_K)$, while those in the fourth correspond to the tensor products of the $\iota_1$ components.

We move to analyze the differentials. For convenience, we work with the universal cover $\pi: \mathbb{R}^2\rightarrow T^2$ of $T^2$. Let $\tilde{\beta}$ be a connected component of the $\pi^{-1}(\beta)$. If $\alpha$ is connected, we let $\tilde{\alpha}$ be a connected component of the $\pi^{-1}(\alpha)$. Otherwise, let $\alpha'$ be a lift of $\alpha$ to $\mathbb{R}^2$, and let $\tilde{\alpha}=\cup_{i\in \mathbb{Z}}t^i (\alpha')$, where $t$ is the horizontal covering translation. Note $\widehat{CFK}(\tilde{\alpha},\tilde{\beta},\pi^{-1}(w),\pi^{-1}(z))=\widehat{CFK}(\alpha,\beta,w,z)$.

Let $x=x_0\otimes x_1$ and $y=y_0\otimes y_1$ be two intersection points of $\alpha$ and $\beta$. Given a holomorphic disk connecting some lift $\tilde{x}$ and $\tilde{y}$ of $x$ and $y$ that contributes to the differential of $\widehat{CFK}(\tilde{\alpha},\tilde{\beta})$, we claim there is a corresponding matching type D and type A operation giving the desired differential via the box-tensor product operation. 

To prove the claim, we first explain how a holomorphic disk as above induces a type A operation in $\widehat{CFA}(P,w,z)$. To do this, we define a \textit{collapsing operation} on $\mathbb{R}^2$ that sends holomophic disks in $\mathbb{R}^2$ to holomorphic disks in $\Sigma$. The collapsing operation is defined to be the composition of the following five operations (see Figure \ref{collapsing operation}):
\begin{itemize}
\item[(Step 1)]Assume the $w$-base point has coordinate $[\epsilon,\epsilon]$ on $[0,1]\times [0,1]$, we hollow an open disk of radius $\sqrt{2}\epsilon$ around every integer point in $\mathbb{R}^2$. 
\item[(Step 2)]Enlarge the holes to a hook shaped region by pushing the boundary, and while doing so, let $\tilde{\alpha}$ move along accordingly. 
\item[(Step 3)]Enlarge the holes one more time to a rounded square by pushing the boundary.
\item[(Step 4)]Collapse the lifts of the second quadrant.
\item[(Step 5)]Collapse the lifts of the fourth quadrant.
\end{itemize}
\begin{figure}
\begin{center}
\includegraphics[scale=0.55]{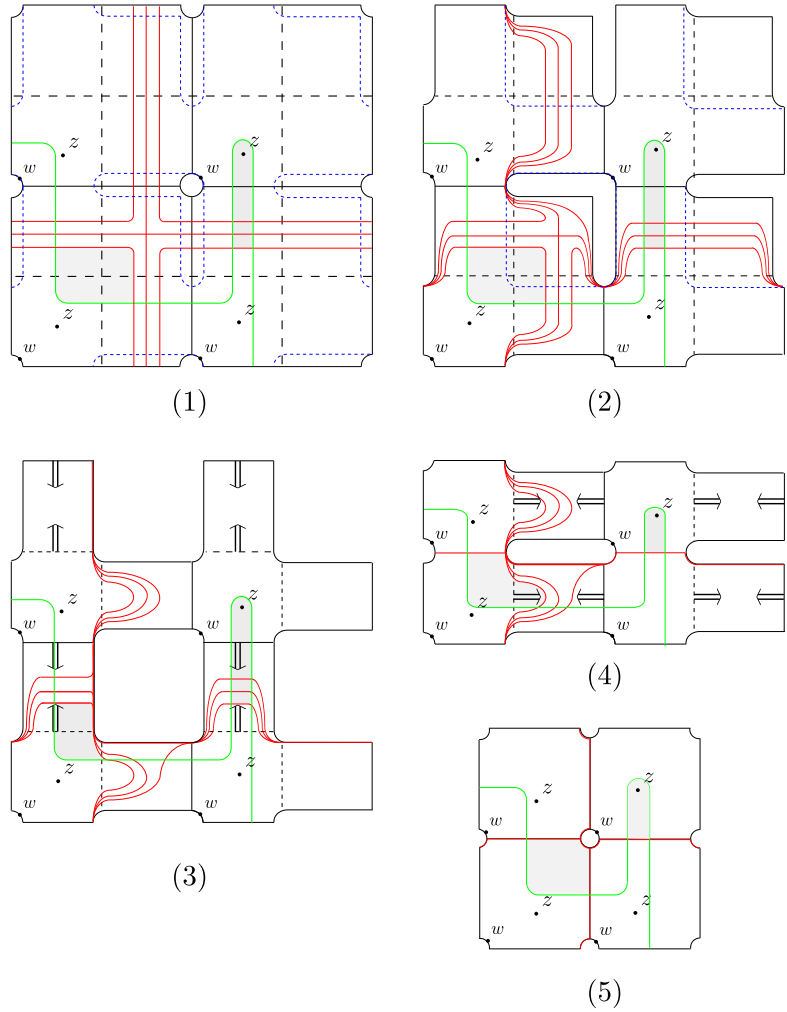}
\caption{Five steps in the collapsing operation.}\label{collapsing operation}
\end{center}
\end{figure}

We choose the enlargement in Step 2 and 3 so that after Step 4 and 5, each hollowed region is bounded by a circle of radius $\sqrt{2}\epsilon$. Denote the resulting space of the collapsing operation by $\tilde{\Sigma}$. 

Note $\tilde{\Sigma}$ can be identified with a covering space of the bordered diagram $(\Sigma, \alpha^a, \beta, w,z)$: the circles are the lifts of the pointed circle $(\partial \Sigma,w)$, $\tilde{\beta}$ is sent to a curve that covers $\beta$, and the horizontal and vertical segments connecting the circles are lifts of $\alpha^a$. 

The arcs on $\tilde{\alpha}$ are sent to arcs on $\tilde{\Sigma}$ that traverse along lifts of $\alpha^a$ and boundary circles according to the following rule:
\begin{itemize}
\item[(1)] For $I\in \{1, 2, 3, 12, 23, 123\}$, a $\rho_I$-arc in the lifts of the first quadrant is sent to $\rho_I$ on the pointed match circle.
\item[(2)] Arcs in the lifts of the second/fourth quadrant are projected horizontally/vertically to the lifts of the $\alpha^a$.  
\end{itemize}
 
Let $\phi: D^2\rightarrow \mathbb{R}^2$ be a holomorphic disk connecting some lift of $x=x_0\otimes x_1$ and $\tilde{y}=y_0\otimes y_1$, further assume $\phi$ does not cross the $w$-base points and has Maslov index one. We will write $D_{\phi}$ for the domain of $\phi$, and use them interchangeably by abusing notations. Let $D_A$ be the image of $D_{\phi}$ under the collapsing operation, this is a domain connecting $x_0$ and $y_0$. We claim $D_A$ determines an embedded disk in $\tilde{\Sigma}$, and the Reeb chords appearing on $\partial D_A$ is given by $\partial_{\tilde{\alpha}}\phi$. Assuming this claim, we see the differential induced by $\phi$ has a correspondent in the box tensor product. Since $D_A$ is an embedded disk, it gives a type A operation in $\widehat{CFA}(P,w,z)$. On the other hand, $D_A$ being embedded implies with the induced orientation, the Reeb chords on $\partial D_A$ are oriented consistently with the convention to produce the immersed curves of a type D module. As $\partial_{\tilde{\alpha}}D_{\phi}$ gives the same sequence of Reeb chords as $\partial D_A$, the type D operation in $\widehat{CFD}(X_K,\mu_K,\lambda_K)$ determined by $\partial_{\tilde{\alpha}}D_{\phi}$ pairs with the type A operation induced by $D_A$. So a differential in $\widehat{CF}(\alpha,\beta)$ corresponds to a differential in $\widehat{CFA}(P,w,z)\boxtimes \widehat{CFD}(X_K,\mu_K,\lambda_K)$.

\begin{figure}
\begin{center}
\includegraphics[scale=0.6]{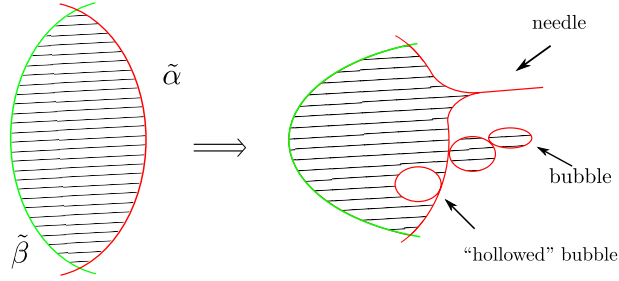}
\caption{Degeneraitions}\label{figure, degenerations}
\end{center}
\end{figure}
We move to prove the aforementioned claim. Suppose the claim is not true, in view of the construction of the collapsing operation, then part of $\partial_{\tilde{\alpha}}\phi$ are pinched together during the collapsing operation, creating needle and bubble degeneration as shown in Figure \ref{figure, degenerations}. So it suffices to prove such degeneration do not happen. To see needle degeneration do not exist, simply note that the tip of a needle would correspond to an $\emptyset$-arrow in $\widehat{CFD}(X_K,\mu,\lambda)$, which contradicts to our assumption that the type-$D$ module is reduced. To see bubble degeneration do not exist, we separate the discussion into two steps. First, we observe that there are no bubble degeneration bounding disks. If not, as the boundary of the bubble consists of $\alpha$-arcs and Reeb chords, the disk it bounds must contain some lifts of the $w$-base point; this would imply $D_{\phi}$ also contains the $w$-base point and hence contradicts to our assumption. Secondly, as we may now assume no bubbles bound disks, the existence of bubbles would imply there is a ``hollowed" bubble: if we orient the bubble counterclockwise, the region $D_A$ appears on its left (see Figure \ref{figure, degenerations}), yet such a bubble again implies that $D_A$ contains the $w$-base point: the boundary of the ``hollowed" bubble corresponds to a loop in $\tilde{\Sigma}$ that consists of $\alpha$-arcs and Reeb chords that do not cross any lifts of the $w$-base point, and any outward region abut it must contain some lift of the $w$-base point (as shown in Figure \ref{figure, no hollow bubble}). This finishes the proof of the claim.

\begin{figure}
\begin{center}
\includegraphics[scale=0.6]{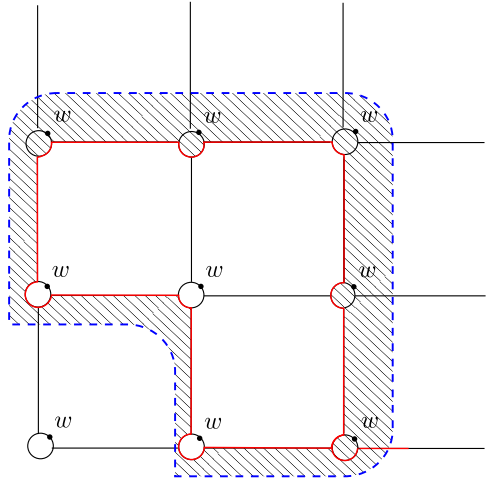}
\caption{Degeneraitions}\label{figure, no hollow bubble}
\end{center}
\end{figure}
Conversely, suppose there is a differential in $\widehat{CFA}(P,w,z)\boxtimes \widehat{CFD}(X_K,\mu_K,\lambda_K)$, then we know there is a disk $D_A$ in $\tilde{\Sigma}$ corresponding to the type A operation, and there is a curve in $\widehat{CFD}(X_K,\mu,\lambda)$ corresponding to the type D operation. Inverting the collapsing operation, we see there is a holomorphic disk giving the desired differential in $\widehat{CFK}(\alpha,\beta)$.

Now we move to prove the isomorphism also preserves the relative Alexander grading. Symmetry of the knot Floer homology would then imply the absolute Alexander grading are also preserved. From now on, we will assume $\alpha$ is connected for convenience. 

Recall from the preliminaries that bordered Floer invariants are graded by certain coset spaces of the enhanced grading group $\widetilde{G}$. For considering the Alexander grading, it suffices to use a simpler group $G_A$, which is obtained from $\widetilde{G}$ by forgetting the Maslov component. Abusing the notation, we will still denote the grading function by $gr$, even though its value are now in (coset spaces of) of $G_A$.

Given $x=x_0\otimes x_1$, $y=y_0\otimes y_1\in \widehat{CFA}(P)\boxtimes\widehat{CFD}(X_K)$, let $\tilde{x}=\tilde{x}_0\otimes \tilde{x}_1$ and $\tilde{y}=\tilde{y}_0\otimes \tilde{y}_1$ be the corresponding lifts. Let $\tilde{P_0}$ be a path on $\tilde{\beta}$ connecting $\tilde{y}_0$ to $\tilde{x}_0$, and $\tilde{P_1}$ be a path on $\tilde{\alpha}$ connecting $\tilde{x}_1$ to $\tilde{y}_1$.  Then $\tilde{P_0}\cup \tilde{P_1}$ bounds a domain $\tilde{B}$ in $\mathbb{R}^2$. Under the covering projection, $\tilde{B}$ gives rise a domain $B'$ in $T^2$; by subtracting or adding copies of $T^2$, we may assume $B'$ does not contain the $w$-base point. Note we can also perform the collapsing operation on $T^2$ to get to $\overline{\Sigma}$, and this gives rise to a domain $B\subset \overline{\Sigma}$. Denote by $\rho(\tilde{P_1})$ the sequence of Reeb chords determined by $\tilde{P_1}$ (the order is induced by the orientation of $\tilde{P_1}$). Note $gr(\partial^{\partial} B)=gr(\rho(\tilde{P_1}))$, and $gr(y_1)=gr(\rho(\tilde{P_1}))^{-1}gr(x_1)$, and $gr(y_0)=\mu^{n_z(B)-n_w(B)}gr(x_0)gr(\partial^{\partial}B)$ (recall $\mu=(0,0;-1)$). Therefore 
\begin{align*}
gr(y_0\otimes y_1)=& \mu^{n_z(B)-n_w(B)}gr(x_0)gr(\partial^{\partial}B)gr(\rho(\tilde{P_1}))^{-1}gr(x_1)\\
=&\mu^{n_z(B)-n_w(B)}gr(x_0\otimes x_1)
\end{align*}
Finally, note $n_z(\tilde{B})-n_w(\tilde{B})=n_z(B)-n_w(B)$. Therefore, the relative Alexander grading induced by pairing the bordered Floer invariants equals that of the Lagrangian Floer chain complex $\widehat{CFK}(\alpha,\beta,w,z)$. 

For the Maslov grading, one only needs to modify the argument in Section 2.3 and 2.4 of \cite{hanselman2018heegaard}. Here we point out the extra cares needed to be taken in this case, and refer the reader to \cite{hanselman2018heegaard} for details. Note the curve $\beta\subset \partial(S^1\times D^2)$ is actually the immersed curve corresponding to a subdiagram of the decorated diagram for $\widehat{CFD}(P, w, z)$. The argument in \cite{hanselman2018heegaard} deals with the Maslov grading in the case when pairing two reduced bordered invariants. The extra care needed for the non-reduced case is understanding the effect of $\emptyset$-arrows: when dealing with such arrows, we degenerate it into a folded line segment, hence both the area contribution and adjusted area contribution would be zero, and the adjusted path contribution is $-1$. With this at mind, the argument in \cite{hanselman2018heegaard} can be adapted to the current setting. It is worth pointing out the immersed curves, as Lagrangian, are in general not embedded and sometimes are even obstructed. Therefore, we recall the definition of Maslov grading difference used in \cite{hanselman2018heegaard} for the convenience of computation.
\begin{defn}
Let $\gamma_0$ and $\gamma_1$ be immersed train tracks in $T^2$, and let $x,y\in \gamma_0\cap \gamma_1$. Suppose there is a path $p_i$ from $x$ to $y$ on $\gamma_i$, $i=0,1$, such that $p_0-p_1$ lifts to a closed loop $l$ in $\mathbb{R}^2\backslash \mathbb{Z}^2$. The Maslov grading difference $m(y)-m(x)$ is twice the number of lattice points enclosed by $l$ (where each point is counted with multiplicity the winding number of $l$) plus $\frac{1}{\pi}$ times the net total rightward rotation along the smooth segments of $l$.
\end{defn}

\end{proof}

\section{Computing the $\tau$-invariant of satellite knots}\label{section,algorithm for computing tau}
In this section, we give a way to compute the $\tau$-invariant by manipulating the pairing diagram for $\widehat{CFK}(P(K))$ when $P$ is a $(1,1)$-pattern.

Recall the Alexander filtration on the hat version knot Floer chain complex induces a spectral sequence, and the $\tau$-invariant can be defined as the Alexander grading of the cycle surviving the infinity page. Each time when we pass from one page to the next, it amounts to cancel differentials that connects elements of the minimal Alexander filtration difference. Such cancellations can be done in the diagram: continuing with the notation in Theorem \ref{Pairing thm}, one can isotope the curve $\beta$ to eliminate pairs of intersection points of minimal filtration difference by pushing the curve across embedded Whitney disks. At the end of such isotopies, only one intersection point is left, which corresponds to the cycle that survives the infinity page. Hence the Alexander grading of this last intersection point is the $\tau$-invariant of the satellite knot. 

In the practice of carrying out this procedure we need to remember the filtration difference of the intersection points. To do this, we introduce the so called \textit{A-buoys}, which will be arrows attached to the $\beta$ curve. To explain how the A-buoys work, we first give the following lemma. 

\begin{lem}\label{lemma, Alexander grading difference is intersection number}
Using the notation in Theorem \ref{Pairing thm}, let $x$, $y$ be two intersection points of $\alpha$ and $\beta$. Let $l$ be an arc on $\beta$ from $x$ to $y$, and let $\delta_{w,z}$ be a straight arc connecting $w$ to $z$. Then $$A(y)-A(x)=l \cdot \delta_{w,z}$$.
\end{lem}
\begin{proof}[Proof of Lemma \ref{lemma, Alexander grading difference is intersection number}]
Let $D$ be a Whitney disk connecting $x$ to $y$ such that $\partial_{\beta}D=-l$. Then $A(y)-A(x)=n_w(D)-n_z(D)=-\partial D \cdot \delta_{w,z}$. Note $\partial_{\alpha}D \cdot \delta_{w,z}=0$ (Figure \ref{figure_pairing}) and hence the lemma follows.
\end{proof}

With this lemma understood, note if we perform an isotopy of $\beta$ by pushing it across an embedded Whitney disk that contains the $z$-base point, the (algebraic) intersections of $\beta$ with $\delta_{w,z}$ are changed. To remedy, we put a right amount of small arrows on $\beta$ whenever such an isotopy is performed, and then when we count the Alexander filtration difference, we count both the algebraic intersection between the corresponding arc on $\beta$ an $\delta_{w,z}$, together with intersections of this arc and these newly added small arrows. These small arrows are the so-called A-buoys.  

\begin{exam}\label{example, tau of Mazur of Trefoil}
We continue considering the satellite knot $M(T_{2,3})$, where $M$ is the Mazur pattern. Note we first cancel some differentials whose filtration difference are one (Left of Figure \ref{figure, tau of Mazur of Trefoil}), and then we are left with five intersection points $x_i$, $i=1,\cdots ,5$ (Right of Figure \ref{figure, tau of Mazur of Trefoil}). Note the filtration difference between $x_3$ and $x_4$ is two as indicated by the A-buoys, while $x_1$ and $x_2$, $x_4$ and $x_5$ has filtration difference one, coming from the intersection with $\delta_{w,z}$. So we cancel the differentials of filtration difference one first, leaving $x_3$ as the remaining intersection point. Once we figure out the relative Alexander grading of $\widehat{HFK}(M(T_{2,3}))$ using the diagram  on the left of Figure \ref{figure, tau of Mazur of Trefoil}, we use the symmetry of the knot Floer homology we can also determine the absolute Alexander grading $x_3$. Finally $\tau(M(T_{2,3}))=A(x_3)=2.$
\begin{figure}
\begin{center}
\includegraphics[scale=0.6]{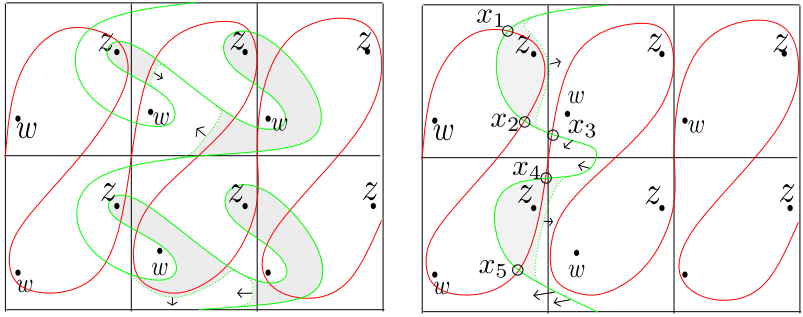}
\caption{Computing the $\tau$-invariant for $M(T_{2,3})$: $x_3$ is the last intersection point left and $\tau(M(T_{2,3}))=A(x_3)=2.$}\label{figure, tau of Mazur of Trefoil}
\end{center}
\end{figure}
\end{exam}
\section{The classification of $(1,1)$-unknot patterns}
We call a pattern $P$ to be an \emph{unknot pattern} if the satellite knot constructed by applying $P$ to the unknot is isotopic to the unknot. Knot Floer homology detects the unknot: a knot $K$ is the unknot if and only if $\widehat{HFK}(K)\cong \mathbb{F}$ \cite{MR2023281}. Exploiting this fact and Theorem \ref{Pairing thm} we give a classification of $(1,1)$ unknot patterns (i.e. unknot patterns that admit a genus-one Heegaard diagram).

We first classify all the genus-one Heegaard diagrams that give rise to unknot patterns.
\begin{thm}\label{classification of unknot pattern Heegaard diagrams}
Nontrivial genus-one doubly-pointed bordered Heegaard diagrams that give rise to unknot patterns are in one-to-one correspondence with pairs of integers $(r,s)$ such that $|r|\geq  1$, and $\gcd(2|r|-1, |s|+1)=1$.
\end{thm}
\begin{figure}[ht!]
\begin{center}
\includegraphics[scale=0.5]{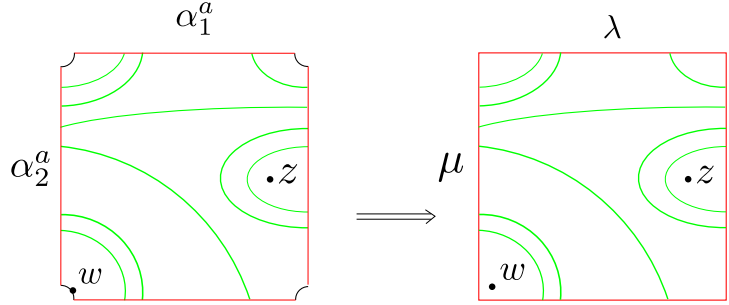}
\caption{The 5-tuple from a bordered Heegaard diagram}\label{figure, convert bordered diagram to 5 tuple}
\end{center}
\end{figure}

\begin{proof}
Recall a genus-one Heegaard diagram $(\Sigma, \alpha^a, \beta, w, z)$ is equivalent to a 5-tuple $(\beta, \mu,\lambda, w, z)\subset T^2$ (Figure \ref{figure, convert bordered diagram to 5 tuple}). So to classify bordered diagrams for unknot patterns, we equivalently classify their corresponding $5$-tuples. 

\begin{lem}\label{lemma, unknot pattern diagram feature}
Let $(\beta, \mu,\lambda, w, z)\subset T^2$ be a 5-tuple constructed from a genus-one Heegaard diagram, then $\beta\cdot \mu=0$ and $\beta\cdot \lambda=\pm 1$, where we arbitrarily orient the curves. In particular, if the $5$-tuple gives rise to an unknot pattern, then up to isotopy, $\lambda$ and $\beta$ has a single intersection point.
\end{lem}
\begin{proof}[Proof of Lemma \ref{lemma, unknot pattern diagram feature}]
Note if we ignore the $z$-base point, then the doubly pointed Heegaard diagram descends to a bordered diagram for the solid torus with standard parametrization of the boundary. Hence if we remove $z$ in the 5-tuple, up to isotopy it is represented by the diagram shown in Figure \ref{figure, no z bordered diagram}. This implies $\beta\cdot \mu=0$, and $\beta\cdot \lambda=\pm 1$.
\begin{figure}[ht!]
\begin{center}
\includegraphics[scale=0.4]{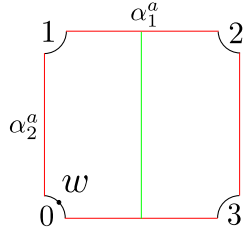}
\caption{Bordered diagram for the solid torus}\label{figure, no z bordered diagram}
\end{center}
\end{figure}
Now assume the $5$-tuple is constructed from an unknot pattern $P$. Pair this diagram with the immersed curve associated to the unknot complement (which is a horizontal line parallel to $\lambda$) using Theorem \ref{Pairing thm}. Then $\widehat{HFK}(P(U))\cong \mathbb{F}$ implies up to isotopy in the complement of $w$ and $z$, $\beta$ can be arranged to intersect $\lambda$ geometrically once.

\end{proof}

\begin{figure}[htb!]
\begin{center}
\includegraphics[scale=0.5]{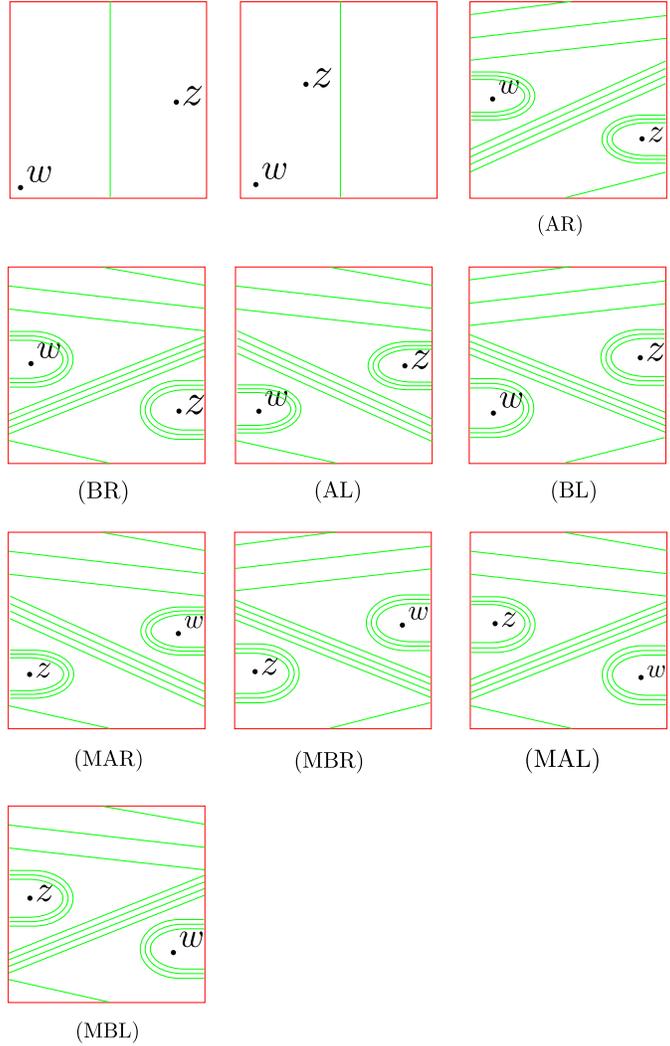}
\caption{Ten types of fundamental regions of bordered Heegaard diagrams for unknot patterns}\label{figure, classifying_fundamental_regions}
\end{center}
\end{figure}

Up to isotopy, assume $\mu$ and $\lambda$ intersect geometrically at one point, cut $T^2$ open along $\mu$, $\lambda$, and call the resulting region \emph{the fundamental region}. Then Lemma \ref{lemma, unknot pattern diagram feature} implies, up to isotopy of $\beta$ in the complement of $w$ and $z$, the diagram of the $5$-tuple in the fundamental region is one of the ten types as shown in Figure \ref{figure, classifying_fundamental_regions}.

We further classify the last eight nontrivial types. Note that the last four types are vertical reflections of the (AR), (BR), (AL), and (BL), so it suffices to classify (AR), (BR), (AL), and (BL). Further note that (AR) and (BR) are horizontal reflections of (AL) and (BL), so it suffices to classify (AL) and (BL).

Note in all the nontrivial cases, the diagram is determined by a pair of numbers: the number of loops around $w$ and $z$, and the number of strands in the middle stripe that separates $w$ and $z$. This is because the rest of the arcs are determined by the condition $\beta \cdot \mu=0$. Yet simply parametrizing each cases by this pair of numbers is not very convenient. Instead, we further group (AR) and (BR) together, call it type (R). Similarly, the other cases are grouped in pair to give (L), (MR), (ML).
\begin{figure}[htb!]
\begin{center}
\includegraphics[scale=0.5]{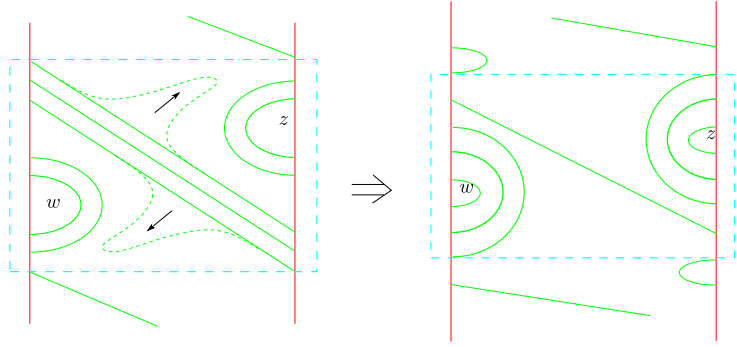}
\caption{Converting (AL) to (BL). The main regions are boxed. By an isotopy, the main region of an (AL) diagram can be converted into one of a (BL) diagram.}\label{figure, convert AL to BL}
\end{center}
\end{figure}

We explain why we group the pairs together using the example of (AL) and (BL). Observe that we may do an isotopy of (AL) as in Figure \ref{figure, convert AL to BL}, so that the region containing the loops and the middle stripe is the same as the corresponding region in case (BL). We call this region the \emph{main region}. Note that by the condition $\beta\cdot \mu=0$, the number of loops $r$ and strands $s$ in the stripe in the main region determines the rest of the diagram. In particular, the pair $(r,s)$ will determine whether the resulting diagram is of type (AL) or (BL). In summary, type (L) diagrams are in one-to-one correspondence with pairs of non-negative integers $(r,s)$ such that $r\geq 1$ and the corresponding $\beta$-curve has a single connected component. 

We characterize the pairs $(r,s)$ whose resulting $\beta$-curve has a single connected component.
\begin{lem}\label{lemma, r,s that gives a single curve}
Given a pair of non-negative integers $(r,s)$ such that $r\geq 1$, then the resulting $\beta$-curve in the main region of a type (BL) diagram with $r$ loops and $s$ bridges has a single connected component if and only if $\gcd(2r-1,s+1)=1$.
\end{lem} 

\begin{figure}[ht!]
\begin{center}
\includegraphics[scale=0.5]{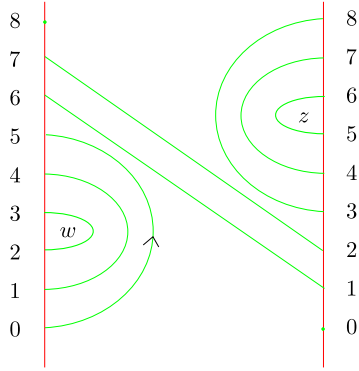}
\caption{The main region of a type (BL) diagram}\label{figure, main region of BL}
\end{center}
\end{figure}

\begin{proof}[Proof of Lemma \ref{lemma, r,s that gives a single curve}]
Label the intersection points of $\beta$ and $\mu$ in the main region by $0,\cdots,2r+s$ (Figure \ref{figure, main region of BL}). Let $a_0=0$, traverse a connected component of $\beta$ in main region starting from the out-most loop around $w$, and denote the intersection points with $\mu$ by $a_1,\cdots, a_n$. Note there is a sequence of numbers $\{\epsilon_i\}\in \{\pm 1\}$ such that $$\epsilon_{i+1}a_{i+1}-\epsilon_i a_i\equiv 2r-1 \ (\text{mod}\ 2(s+2r)).$$
In fact, we may take $\epsilon_i$ to be sign of the intersection of $\beta$ and $\mu$ at $a_i$. 
Now $\beta$ has a single connected component if and only if the subgroup generated by $2r-1$ in $\mathbb{Z}_{2(s+2r)}$ contains the element $0$, $1$ or $-1$, ..., $2r+s$ or $-(2r+s)$. This means $2r-1$ generates $\mathbb{Z}_{2(s+2r)}$. Hence this is equivalent to $\gcd(2r-1,2(s+2r))=1$, which is same as $\gcd(2r-1,s+1)=1$.
\end{proof}
Therefore, each of the cases (L), (R), (ML), and (MR) are in one-to-one correspondence with such pairs $(r,s)$. To uniformly parametrizes all four cases, we allow $r$ and $s$ to be both positive and negative. We set the signs of the parameters $(r,s)$ corresponding to cases (R), (L), (MR), and (ML) to be $(+,+)$, $(+,-)$, $(-,+,)$, and $(-,-)$ respectively.     
\begin{figure}[htb!]
\begin{center}
\includegraphics[scale=0.4]{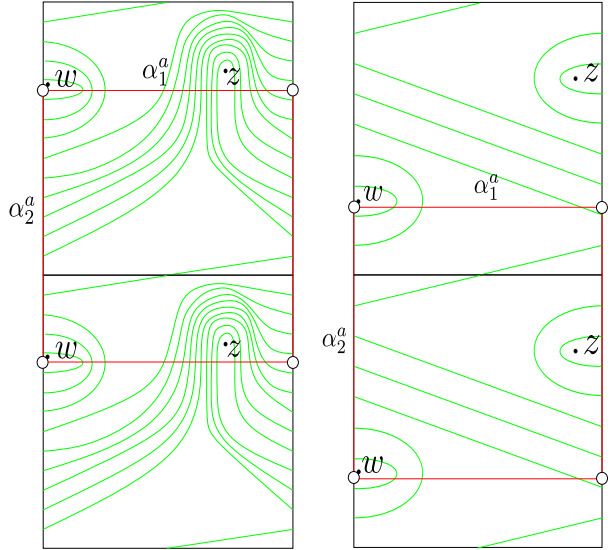}
\caption{Converting diagrams of 5-tuple into bordered Heegaard diagrams. On the left, a type (AR) diagram and its bordered Heegaard diagram. On the right, a type (BL) diagram and its bordered Heegaard diagram. }\label{figure, convert 5 tuple to bordered}
\end{center}
\end{figure}

Finally, one can covert the diagram in the fundamental region to a bordered Heegaard diagram (See Figure \ref{figure, convert 5 tuple to bordered}for examples). This finishes the proof of Theorem \ref{classification of unknot pattern Heegaard diagrams}. 
\end{proof}

Now we recognize the patterns from the doubly pointed diagrams. Eventually, we will identify such patterns as those obtained from \emph{2-bridge links}. All 2-bridge links (knots) admit a presentation called the \emph{Schubert normal form}. Such a normal form is parametrized by a pair of coprime integers $p$ and $q$ such that $p>0$, $0< |q| < \frac{p}{2}$, and is denoted by $b(p,q)$. (See Figure \ref{figure, 2 bridge links} for an example.) In general, $b(p,q)$ is the mirror image of $b(p,-q)$, $b(p,q)$ is a two-component link if and only if $p$ is even, and $b(p,q)$, $b(p',q')$ are isotopic to each other if and only if $p=p'$ and $q'\equiv q^{\pm 1}\ (\text{mod} \ p )$.

\begin{figure}[ht!]
\begin{center}
\includegraphics[scale=0.38]{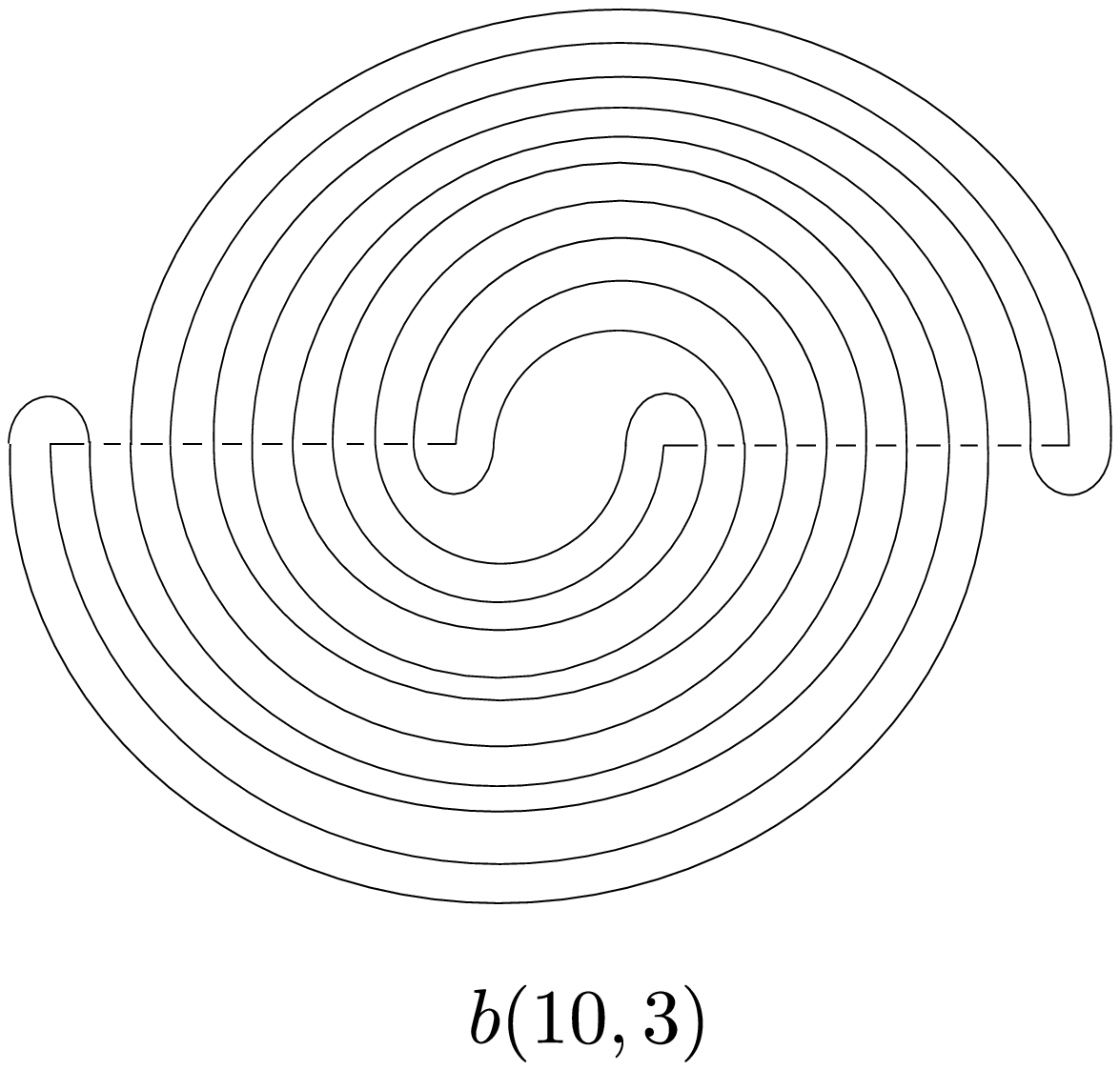}
\caption{The 2-bridge link $b(10,3)$.}\label{figure, 2 bridge links}
\end{center}
\end{figure}

\begin{thm}\label{Theorem, (1,1)-unknot pattern diagram gives two bridge pattern}
Let $P$ be a $(1,1)$-unknot pattern obtained from a genus-one doubly-pointed bordered diagram of parameter $(r,s)$. Then the link consists of $P$ and the meridian of the solid torus is the 2-bridge link $b(2|s|+4|r|, \epsilon(r)(2|r|-1))$. Here $\epsilon(r)$ is the sign function of $r$.  
\end{thm}

\begin{proof}
The pattern knot $P$ can be drawn on the torus by an arc connecting $w$ to $z$ in the complement of the $\beta$-curve, and then an arc in the complement of $\alpha^a$. Viewing in the fundamental region, $P$ has a diagram consisting of two bundles of $|r|-1$ loops and a stripe of $|s|+1$ arcs (Figure \ref{figure, pattern from bordered diagram}). Such a $P$ together with the meridian of the solid torus is the 2-bridge link $b(2|s|+4|r|, \epsilon(r)(2|r|-1))$ according to Lemma 4.1 of \cite{MR2708610}.
\end{proof}

\begin{figure}[htb!]
\begin{center}
\includegraphics[scale=0.3]{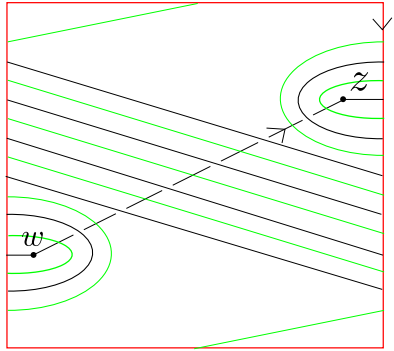}
\caption{The pattern knot. Note such a pattern is determined by the number of loops and strands in the middle stripe of the diagram.}\label{figure, pattern from bordered diagram}
\end{center}
\end{figure}

\section{$\tau$-invariant and two-bridge patterns}
Recall from the previous section that a two-bridge link $b(p,q)$ gives rise to a $(1,1)$-unknot pattern $P$. In this section, we derive a formula for $\tau(P(T_{2,3}))$ and $\tau(P(-T_{2,3}))$ in terms of $p$ and $q$. The formula involves two functions; one is a quantity $\sigma$ associated to some Heegaard diagram for such a pattern, and the other is the winding number $w(P)$. We explain these quantities below.

First a remark on the convention, from the previous section we know bordered diagrams for such patterns are parametrized by a pair of integers, and in this section, we will restrict attentions to diagrams of type (L), given by pairs $(r,-s)$ with $r,s\in \mathbb{Z}$, $r>1$, and $\gcd(2r-1,s+1)=1$. Note generality is not lost by this restriction as we may pass to mirror image of $P(K)$ when $P$ is not of type (L).

We explain how we associate a quantity $\sigma$ to a genus-one Heegaard diagram corresponding to an unknot pattern. 
\begin{defn}
Given a type (L) diagram $H=H(r,-s)$, $\sigma(H)$ is defined to be the algebraic intersection number of the arc on $\beta$ outside of the main region and a left push-off of meridian $\mu$ (still denoted by $\mu$). Here we orient $\beta$ and $\mu$ compatibly so that if we isotope $\mu$ to $\beta$ (disregarding $w$ and $z$), then the orientation of $\beta$ is induced from that of $\mu$ (See Figure \ref{figure, definition of sigma}). 
\end{defn} 

\begin{figure}[htb!]
\begin{center}
\includegraphics[scale=0.5]{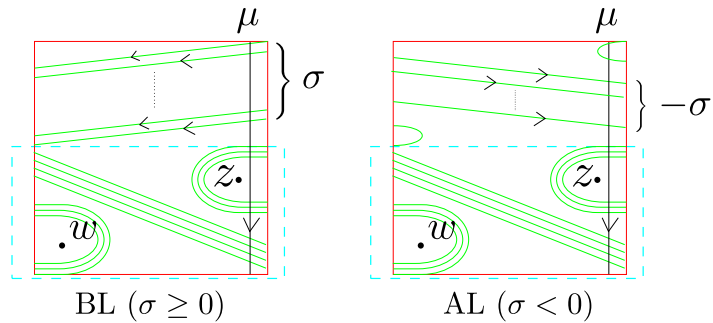}
\caption{Seeing $\sigma(H)$ diagramatically. Note the boxed region is the main region.}\label{figure, definition of sigma}
\end{center}
\end{figure}

For convenience, we call the arc out of the main region the \emph{dependent arc}, as it is determined by the number of caps $r$ around the $w$- and $z$-base points and the number of stripes $s$ in the main region. We also remark that it is necessary to use a push-off in the above definition so that the end points of the dependent arc is not on the push-off meridian, and actually a more precise definition can be given using homology class in some relative homology group, but here we would rather stick with this more straightforward and pictorial explanation.  

Note depending on the parity and sign of $\sigma$, one can draw the 5-tuple diagram corresponding to $H$ in a symmetric way as shown in Figure \ref{figure, symmetric diagram}.

\begin{figure}[htb!]
\begin{center}
\includegraphics[scale=0.55]{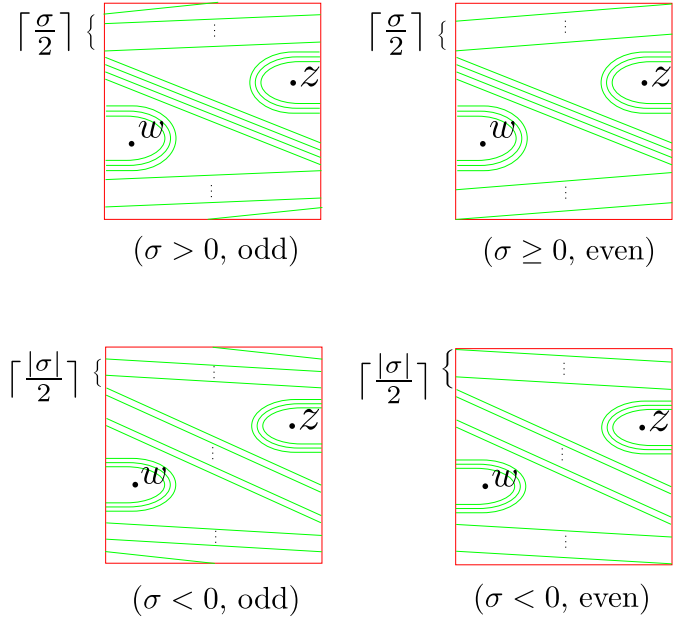}
\caption{Symmetric diagrams for $H(r,-s)$ in the fundamental region.}\label{figure, symmetric diagram}
\end{center}
\end{figure}

Next, we give a closed formula for $\sigma$.
\begin{prop}\label{Proposition, close formula for sigma}
$$\sigma(H(r,-s))=\sigma(s+1,2r-1)=\sum_{i=1}^{2r-2}(-1)^{\lfloor\frac{i(2r-1)}{s+1} \rfloor}$$
\end{prop}
\begin{proof}
Note since $\beta \cdot \mu=0$, we may equivalently study the arc on $\beta$ inside the main region, which we denote by $l$. Note $l\cdot\mu$ is can be described by a simple formula:
\begin{equation}\label{equation, signature}
l\cdot \mu=1-\sum_{i=1}^{2r+s-1}(-1)^{\lfloor\frac{i(2r-1)}{2r+s} \rfloor}
\end{equation}
We explain where the terms in the above formula come from. First note there are a total $2r+s$ many intersection points between $l$ and $\mu$, which we label from $0$ to $2r+s-1$ starting from top to bottom. Now as we traverse along $l$ downwards, the first intersection point is positive, and hence contributes 1 to the right hand side of Equation \ref{equation, signature}. As we move on, the next intersection is negative. Note consecutive intersection points on $l$ will differ by $2r-1$ units modulo $2r+s$ and the further intersections change sign only when we need to make a turn along some caps, which is captured by the floor function. 

Now $\beta\cdot \mu=0$ implies $\sigma(s+1,2r-1)+1-\sum_{i=1}^{2r+s-1}(-1)^{\lfloor\frac{i(2r-1)}{2r+s}\rfloor}=0$. Therefore, 
\begin{align*}
\sigma(s+1,2r-1)&=\sum_{i=1}^{2r+s-1}(-1)^{\lfloor\frac{i(2r-1)}{2r+s}\rfloor}-1\\
&=\sum_{i=1}^{2r-2}(-1)^{\lfloor \frac{i(2r-1)}{s+1} \rfloor},
\end{align*}
where the last equality is proved in Lemma 3.9 of \cite{chen2017alexander}.
\end{proof}

We move to consider the winding number $w(P)$. In order to remove ambiguity of sign when speaking of the winding number, firstly we fix the convention on the orientations of the relevant objects: we orient the meridian $\mu$ of the solid torus so that in the 5-tuple diagram, $\mu$ is oriented downwards, and the pattern knot $P$ is oriented so that the short arc connecting $w$ to $z$ in the complement of $\alpha^a$ is oriented from $w$ to $z$ (Figure \ref{figure, pattern from bordered diagram}).  By pushing $P$ in the interior of $S^1\times D^2$, we set $w(P)=lk(P,\mu)$.

\begin{lem}\label{Lemma, winding number equals intersection of beta and short arc}
Equip $\beta$ with the orientation obtained by isotoping $\mu$ to $\beta$. Recall $\delta_{w,z}$ denotes the straight arc connecting $w$ to $z$ within the fundamental domain of the 5-tuple. We have $w(P)=\beta\cdot \delta_{w,z}$
\end{lem} 
\begin{proof}
Since $\beta$ is isopotic to $\mu$, $lk(P,\beta)=lk(P,\mu)$. Figure \ref{figure, pattern from bordered diagram} shows we have a diagram of $P$ and $\beta$ such that the only intersections between $P$ and the meridian disk bounded by $\beta$ occur on the arc obtained by pushing $\delta_{w,z}$ into the solid torus. Therefore, $\beta\cdot \delta_{w,z}=lk(\beta,P)=w(P)$.
\end{proof}

\begin{figure}[htb!]
\begin{center}
\includegraphics[scale=0.4]{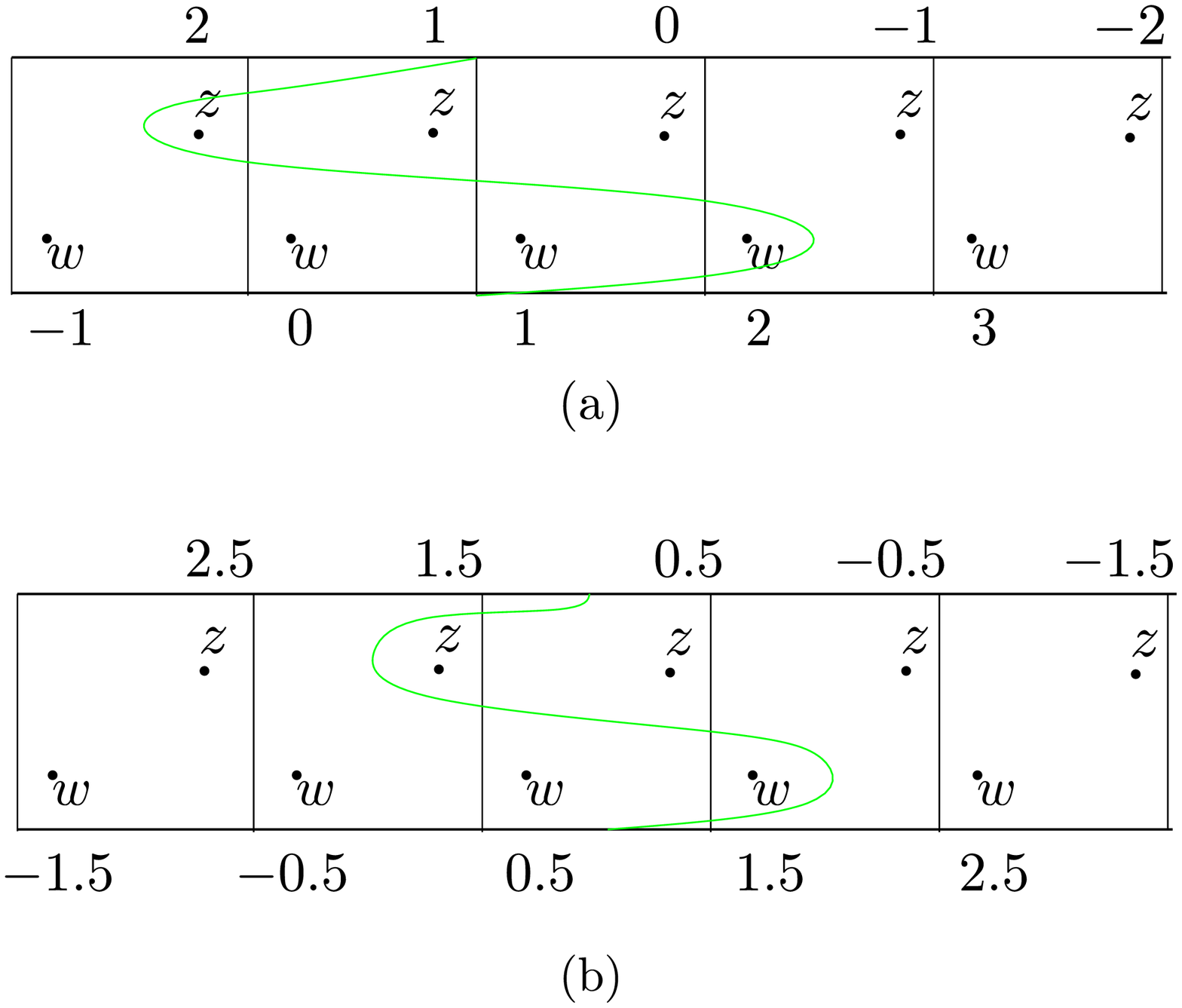}
\caption{(a) shows the case in which $\sigma$ is even, and in the this example $O_w(H)=O_z(H)=2$; (b) shows the case in which $\sigma$ is odd, and in the this example $O_w(H)=O_z(H)=1.5$.   }\label{figure, O_w and O_z}
\end{center}
\end{figure}

We shall need another diagrammatic description of the winding number. Start with a symmetric diagram for $H(r,-s)$, lift the $\beta$ curve to the covering space $S^1\times \mathbb{R}$ of $T^2$ corresponding to the subgroup of $\pi_1(T^2)$ generated by the longitude $\lambda$. Label the lifts of the base point $w$ as in Figure \ref{figure, O_w and O_z}. The rule is: if $\sigma(H)$ is even, then the first $w$ base point to the right of $\beta\cap \lambda$ is labeled by $1$, if $\sigma$ is odd, we label the corresponding point by $1.5$, and in both cases, the labels increase by $1$ as we move from one base point to the next from left to right. Note the lift of the $\beta$ curve separates $S^1\times \mathbb{R}$ into two regions,which we call by the left region and right region. We define $O_w(H)$ to equal to the label of the right-most $w$ base point contained in the left region. One can similarly define $O_z(H)$. See Figure \ref{figure, O_w and O_z} for an example.

\begin{prop}\label{proposition, winding equal to 2 O_w }
Let $P$ be a pattern obtained from the bordered Heegaard diagram $H(r,-s)$, then $w(P)=O_w(H)+O_z(H)=2O_w(H)$.
\end{prop}
\begin{proof}
It is clear $O_w(H)=O_z(H)$ in view of the symmetry of the diagram. Hence it is left to show $w(P)=O_w(H)+O_z(H)$. By Lemma \ref{Lemma, winding number equals intersection of beta and short arc}, $w(P)=\beta\cdot \delta_{w,z}$, and we may equivalently consider this intersection in $S^1\times \mathbb{R}$: let $\tilde{\beta}$ be a single lift of $\beta$, and $\pi^{-1}(\delta_{w,z})$ be the preimage of $\delta_{w,z}$ of the covering map $\pi$. Then $\beta\cdot \delta_{w,z}=\tilde{\beta}\cdot\pi^{-1}(\delta_{w,z})$. The proposition then follows from the following observation: Let $a_{w,z}$ be a connected component of $\pi^{-1}(\delta_{w,z})$, then $\tilde{\beta}\cdot a_{w,z}=0$ if both end points of $a_{w,z}$ are contained in the left region or the right region. $\tilde{\beta}\cdot a_{w,z}=1$ if the $w$ end point of $a_{w,z}$ is in the left region, while the $z$ end point is on the right. Otherwise, $\tilde{\beta}\cdot a_{w,z}=-1$.
\end{proof}
\begin{rmk}
We remark there is also a closed formula for the winding number in interest of computation.For the pattern corresponding to the two-bridge link $b(p,q)$ where $q>0$, the winding number is equal to 
\begin{equation}\label{Equation, closed formula for winding number}
w(p,q)=\sum_{k=0}^{\lfloor \frac{p-2}{2} \rfloor}(-1)^{\lfloor \frac{(2k+1)q)}{p} \rfloor}
\end{equation}
We skip the proof for this formula and remark that it is similar to the proof of Proposition \ref{Proposition, close formula for sigma}.
\end{rmk}

For the proof of Theorem \ref{Theorem, tau invariant formula}, we will partially carry out the algorithm in Section \ref{section,algorithm for computing tau}: Do isotopies that cancel pairs of intersection points whose Alexander grading difference is one, after that we can read off the $\tau$-invariant from the pairing diagram. More concretely, we will push the caps around the $z$-base point off one by one, until this cannot be done any more (See Example \ref{example, tau of Mazur of Trefoil}). 

\begin{figure}[htb!]
\begin{center}
\includegraphics[scale=0.45]{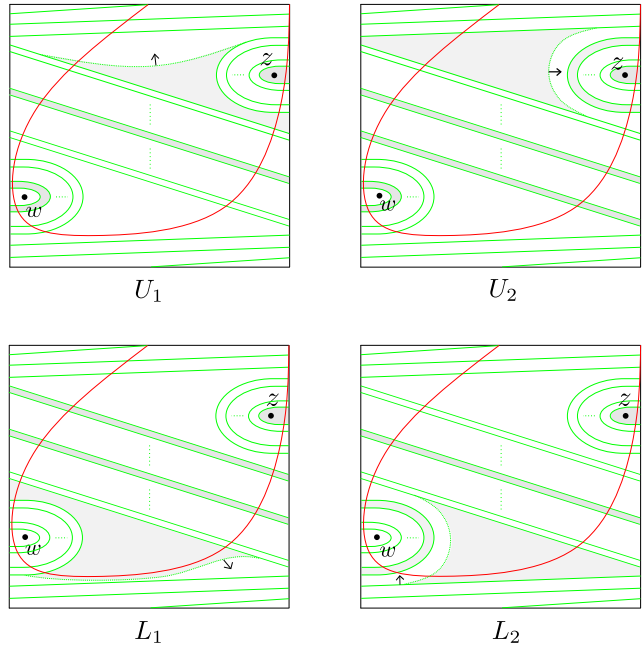}
\caption{Pushing $\beta$ along the grey Whitney disks can possibly end up with four situations.}\label{figure, four possible endings of a z-cap removal}
\end{center}
\end{figure}

\begin{proof}[Proof of Theorem \ref{Theorem, tau invariant formula}]
For the ease of exposition, we give details in the case when the $\sigma(p,q)$ is odd, $|\sigma(p,q)|\geq 3$ and the companion knot is $T_{2,3}$. We remark that similar reasoning would work in other cases. 

We begin with the case when $\sigma$ is odd and $\sigma\geq 3$. Note when the companion knot is the trefoil knot, in the universal cover of the pairing diagram, only two rows contains the intersection points. We refer to them as the upper row and the lower row respectively. 

First, we examine the isotopy of pushing the z-caps off in the lower row. Push the innermost z-cap off and cancel intersection pairs as many as possible, then the $\beta$ curve could possibly end with one of the four cases as shown in Figure \ref{figure, four possible endings of a z-cap removal}, which we call $U_1$, $U_2$, $L_1$, and $L_2$. Note $U_1$ increase the algebraic intersection number of the dependent part of $\beta$ with $\mu$, while $U_2$ decreases. Same observation  
applies to $L_1$ and $L_2$.

Note the A-buoys occur in $U_1$, $L_1$, $L_2$ are out of the main region (recall this is the region containing the caps and the stripes in the middle), and hence will not get involved in the next round of the z-cap removal. The A-buoy which comes from $U_2$ will neither get involved if further z-cap removals end with $U_1$, $L_1$, or $L_2$. The only difference is if another $U_2$ happens, then it creates a Whitney disk which connects a pair of points whose Alexander grading difference is two (see Figure \ref{figure, filtration 2 cap from twice U_2}).

\begin{figure}[htb!]
\begin{center}
\includegraphics[scale=0.5]{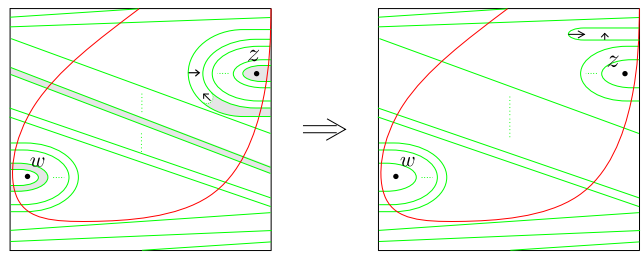}
\caption{Two endings with $U_2$ creates a Whitney disk having filtration difference equal to two.}\label{figure, filtration 2 cap from twice U_2}
\end{center}
\end{figure}

\begin{figure}[htb!]
\begin{center}
\includegraphics[scale=0.6]{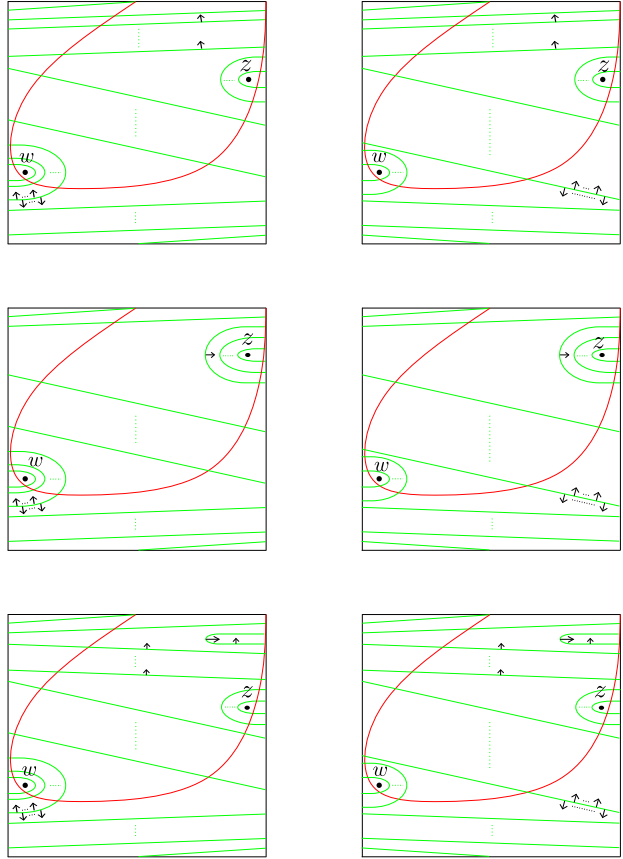}
\caption{Six types of possible situations occur in the process of removing z-caps}\label{figure, six possibilities during cap removal}
\end{center}
\end{figure}

Note in the process of repeatedly carrying out the z-cap removals, the effect of $L_1$ and $L_2$ cancel each other: one increases the algebraic intersection number between the dependent part of $\beta$ and $\mu$ while the other decreases, and the A-buoys would come in different directions and hence offset each other. The same observation apply to $U_1$ and $U_2$.

With these observations at hand, note during the process of the doing the z-cap removals we will be seeing one of the following six situations in Figure \ref{figure, six possibilities during cap removal}; one can check if a further z-cap removal is done to one of the six diagrams, the resulting diagram is still one of them. The process of removing z-caps will end with one of the four types of diagrams as shown in Figure \ref{figure, bottom four endings}.

\begin{figure}[htb!]
\begin{center}
\includegraphics[scale=0.6]{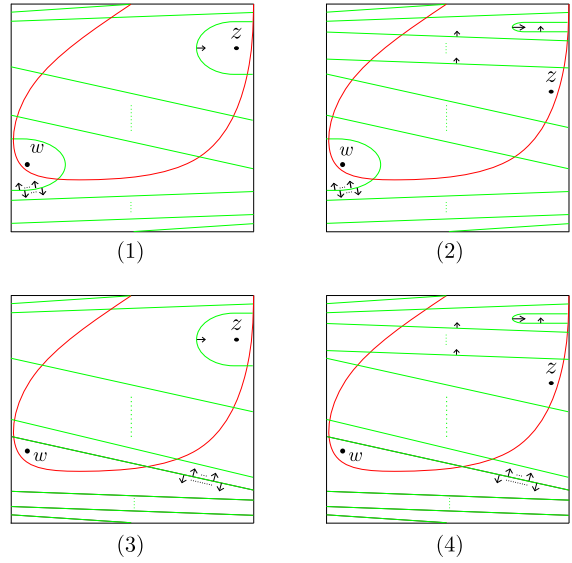}
\caption{Possible endings in the lower row when $\sigma\geq 3$.}\label{figure, bottom four endings}
\end{center}
\end{figure}

The same reasoning can be applied to the analyze the isotopy in the upper row. At the end of this process, the diagram would look like one of the four cases shown in Figure \ref{figure, top four endings}.

\begin{figure}[htb!]
\begin{center}
\includegraphics[scale=0.5]{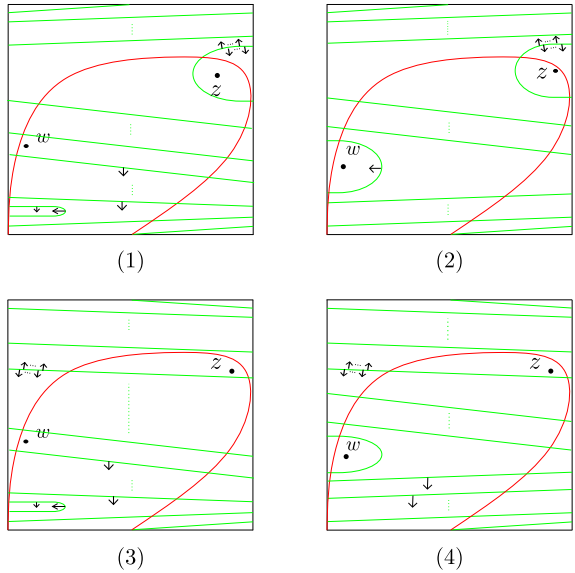}
\caption{Possible endings in the upper row when $\sigma\geq 3$.}\label{figure, top four endings}
\end{center}
\end{figure}
We move to combine the diagrams from Figure \ref{figure, bottom four endings} and Figure \ref{figure, top four endings}. Note if we allow isotopies of the $\beta$-curve that eliminates Whitney disks connecting intersection pairs with Alxander filtration greater than or equal to two, then the ending diagram in the upper row should be isotopic to the one in the lower row. With this understood, one can see there are three type of possibilities of how the simplified pairing diagram looks like Figure \ref{figure, sigma positive top 12+bottom 12}, Figure \ref{figure, sigma positive top 12+bottom 34}, and Figure \ref{figure, sigma positive top 34+bottom 12}. In these figures, the dotted caps stand for those having two A-buoys on the tip. From Figure \ref{figure, six possibilities during cap removal}, one can see such caps would appear in the lower row to ensure that the first turn as we traverse along the $\beta$ curve downwards would appear after it crosses $\mu$ from right to left at least $\lfloor\frac{\sigma}{2} \rfloor$ many times. Similar observation applies to the appearance of such caps in the upper row. 

We move to determine which intersection point has the Alexander grading that equals to the $\tau$-invariant case by case. 

\begin{figure}[htb!]
\begin{center}
\includegraphics[scale=0.45]{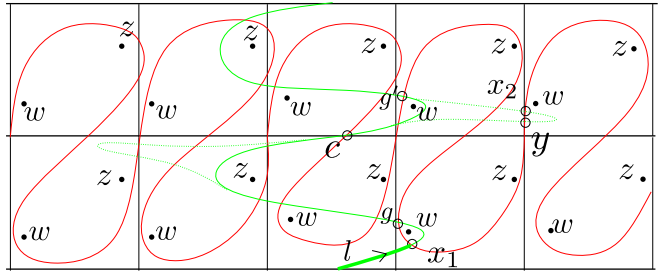}
\caption{Combining Figure \ref{figure, top four endings} (1) or (2) and Figure \ref{figure, bottom four endings} (1) or (2)}\label{figure, sigma positive top 12+bottom 12}
\end{center}
\end{figure}
The case given in Figure \ref{figure, sigma positive top 12+bottom 12} comes from combining Figure \ref{figure, top four endings} (1) or (2) and Figure \ref{figure, bottom four endings} (1) or (2). If the dotted cap does not appear, then $\tau=A(x_1)$, as $x_1$ is an intersection point with neither incoming nor outgoing differentials. If the dotted cap appears, the relevant component of the chain complex consists of three intersection points: $x_1$, $x_2$, and $y$. The differentials are $x_1\rightarrow y$, $x_2 \rightarrow y$. Therefore, $\tau=\max(A(x_1),A(x_2)).$ In the latter case, we claim $A(x_1)\geq A(x_2)$. To see this, note $A(g)-A(x_1)=-1$, $A(g')-A(g)=\beta\cdot\delta_{w,z}=-w(P)$, and $A(x_2)-A(g')\leq 1$ in view of the A-buoys in Figure \ref{figure, top four endings}(1). Since in this case we have $w(P)\geq 1$ by Proposition \ref{proposition, winding equal to 2 O_w } (note $O_w(H)$ is not changed under this z-cap removals), $A(x_2)\leq A(g')+1=A(g)-w(P)+1\leq A(x_1)$. Therefore, $\tau=A(x_1)$ in this case. 
\begin{figure}[htb!]
\begin{center}
\includegraphics[scale=0.5]{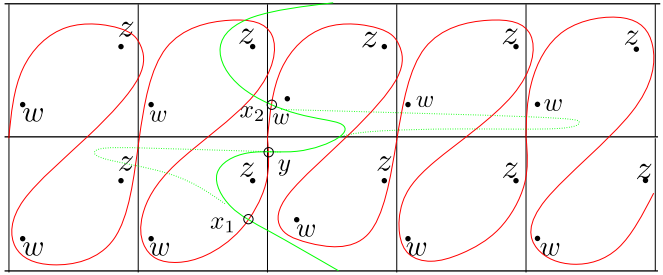}
\caption{Combining Figure \ref{figure, top four endings} (1) and Figure \ref{figure, bottom four endings} (3) or (4)}\label{figure, sigma positive top 12+bottom 34}
\end{center}
\end{figure}

The case given in Figure \ref{figure, sigma positive top 12+bottom 34} comes combining Figure \ref{figure, top four endings} (1) and Figure \ref{figure, bottom four endings} (3) or (4). Similarly we have $\tau=\max(A(x_1),A(x_2))$. Note in this case, $A(x_2)-A(x_1)=\beta \cdot \delta_{w,z}=-w(P)$, where $w(P)$ stands for the winding number of $P$. Also note by Proposition \ref{proposition, winding equal to 2 O_w }, $w(P)\leq 0$. Hence $\tau=A(x_2)=A(x_1)-w(P)$. 

\begin{figure}[htb!]
\begin{center}
\includegraphics[scale=0.5]{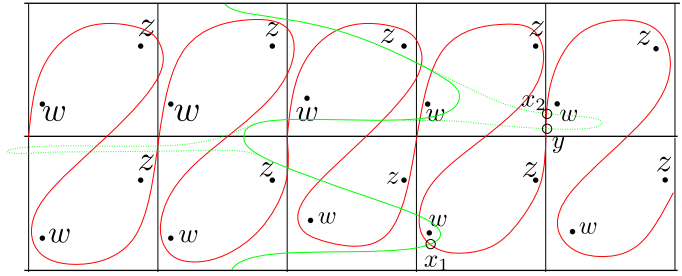}
\caption{Combining Figure \ref{figure, top four endings} (3) or (4) and Figure \ref{figure, bottom four endings} (1) or (2)}\label{figure, sigma positive top 34+bottom 12}
\end{center}
\end{figure}

The case given in Figure \ref{figure, sigma positive top 34+bottom 12} comes combining Figure \ref{figure, top four endings} (3) or (4) and Figure \ref{figure, bottom four endings} (1) or (2). A similar analysis shows $\tau=A(x_1)$. 

In view of the discussion above, it suffices to determine $A(x_1)$. Note $x_1$ is the first intersection of $\beta$ and $\alpha$ as we traverse upwards along $\beta$ in the simplified pairing diagram; denote this oriented arc starting from $\beta\cap\lambda$ and going upward to $x_1$ by $l$ (see Figure \ref{figure, sigma positive top 12+bottom 12}), and denote the corresponding arc in the original pairing diagram by $\tilde{l}$. Note $l\cdot\mu=\lfloor \frac{w}{2} \rfloor$ according to Proposition \ref{proposition, winding equal to 2 O_w }, and the A-buoys on $l$ contribute in total to the Alexander grading by $\lceil \frac{\sigma}{2} \rceil-\lfloor \frac{w(P)}{2} \rfloor=\tilde{l}\cdot \delta_{w,z}.$ This last equation can be seen by understanding the effect of $L_1$ and $L_2$ (Figure \ref{figure, four possible endings of a z-cap removal}) and apply an inductive argument: suppose only a single $L_1$ occurs throughout; in the begining, as we traverse along $\beta$ upwards, it intersects $\mu$ for $\lceil \frac{\sigma}{2}\rceil$ times until we reaches the first intersection point, and each $L_1$ increases this intersection by 1, and the corresponding $\tilde{l}\cdot \delta_{w,z}=-1$; $L_2$ has an opposite effect.

Let $c$ be the intersection of $\beta$ and $\alpha$ lying in the center (See Figure \ref{figure, sigma positive top 12+bottom 12}). We have, $$A(c)-A(x_1)=(\beta-\tilde{l})\cdot \delta_{w,z}=-w(P)-(\lceil \frac{\sigma}{2} \rceil-\lfloor \frac{w(P)}{2} \rfloor)=-\frac{w(P)+\sigma(P)}{2}-1$$
We claim $A(c)=0$. Assuming this claim, we have $A(x_1)=\frac{w(P)+\sigma}{2}+1$. It is then straightforward to see $\tau(P(K))=\frac{|w(P)|+\sigma(P)}{2}+1$ when $\sigma\geq 3$ and $\sigma$ is odd.

We now justify our claim on the Alexander grading of $c$. 
\begin{figure}[ht!]
\begin{center}
\includegraphics[scale=0.45]{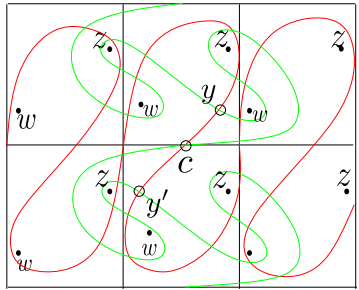}
\caption{The Alexander grading of the center intersection point is $0$.}\label{figure, Alexander grading of the center intersection}
\end{center}
\end{figure}
\begin{proof}[Proof of the claim $A(c)=0$]
It is better to have a concrete example in mind, see Figure \ref{figure, Alexander grading of the center intersection}. Note $c$ corresponds to the center point in the symmetric minimal intersection diagram: if we rotate such pairing diagram about $c$ for an angle $\pi$, the diagram goes back to itself. We pair the intersection point with its symmetric counterpart. Take a pair of such intersection points $y$ and $y'$, let $l_{cy}$ and $l_{cy'}$ denote the arc on $\beta$ from $c$ to $y$ and $y'$ respectively. Note $$A(y)-A(c)=l_{cy}\cdot \delta_{w,z}.$$
\begin{align*}
A(y')-A(c)&=l_{cy'}\cdot \delta_{w,z}\\
&=-l_{cy'}\cdot (-\delta_{w,z})\\
&=-l_{cy'}\cdot \delta_{z,w}\\
&=-l_{cy}\cdot \delta_{w,z}\\
&=A(c)-A(y),
\end{align*}
where in the fourth equality we used the symmetry of the diagram. Therefore the Alexander grading of elements of $\widehat{HFK}(P(T_{2,3}))$ is symmetric about $A(c)$, and hence $A(c)=0$ by the convention of how we grade knot Floer homology groups. This finishes the proof of the claim.
\end{proof}

We move to consider the case when $\sigma\leq-3$ and odd.
Similar analysis reveals how the pairing diagrams would like at the end of z-cap removals; the upper row is shown in Figure \ref{figure, top three endings sigma negative} and the lower row is shown in Figure \ref{figure, lower three endings sigma negative}, both have three cases. Again we combine Figure \ref{figure, top three endings sigma negative} and Figure \ref{figure, lower three endings sigma negative}, and discuss the corresponding value of the $\tau$-invariant. 
\begin{figure}[htb!]
\begin{center}
\includegraphics[scale=0.5]{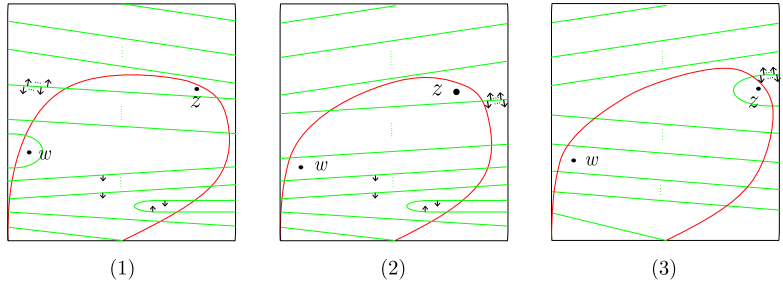}
\caption{Possible endings in the upper row when $\sigma\leq -3$.}\label{figure, top three endings sigma negative}
\end{center}
\end{figure}

\begin{figure}[htb!]
\begin{center}
\includegraphics[scale=0.5]{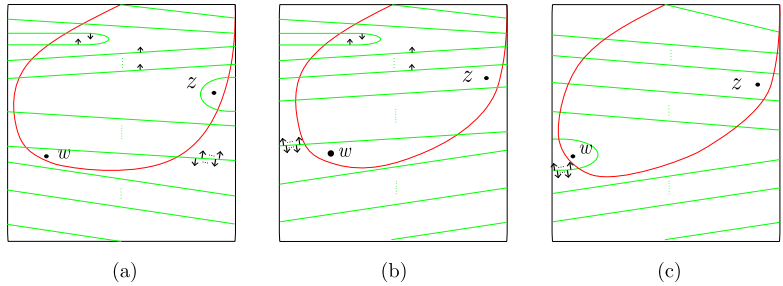}
\caption{Possible endings in the lower row when $\sigma\leq -3$.}\label{figure, lower three endings sigma negative}
\end{center}
\end{figure}

First, if the lower row is of type (a) in Figure \ref{figure, lower three endings sigma negative}, then the upper row could be of type (2) or type (3) in Figure \ref{figure, top three endings sigma negative}. Figure \ref{figure, a+2 sigma negative} shows the case when the lower row is of type (a) and the upper row is of type (2). In this case, $0\geq w(P)>\sigma+2$. An analysis of the differential as above shows $\tau=\max(A(x_2),A(x_1))$. Note we have $A(x_2)=0$ (since $A(x_2)-A(c)=0$) and $A(x_1)=\frac{w(P)+\sigma}{2}+1\leq 0$. Therefore, $\tau=0$. Figure \ref{figure, a+3 sigma negative} shows the case when the lower row is of type (a) and the upper row is of type (3). In this case, $w(P)\leq\sigma+2$. Again $\tau=\max(A(x_2),A(x_1))$. While we still have $A(x_1)=\frac{w(P)+\sigma}{2}+1$, $A(x_2)=A(x_1)-w(P)=\frac{-w(P)+\sigma}{2}+1$. Therefore, $\tau=\frac{|w(P)|+\sigma}{2}+1$.
\begin{figure}[htb!]
\begin{center}
\includegraphics[scale=0.5]{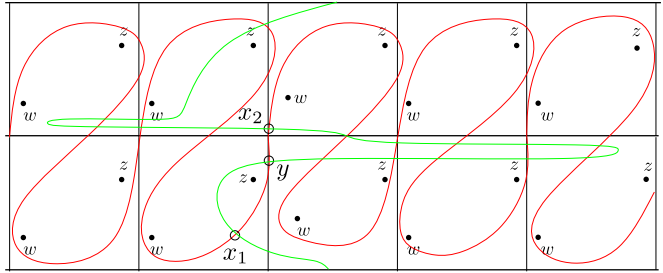}
\caption{Lower type (a), upper type (2).}\label{figure, a+2 sigma negative}
\end{center}
\end{figure}
\begin{figure}[htb!]
\begin{center}
\includegraphics[scale=0.5]{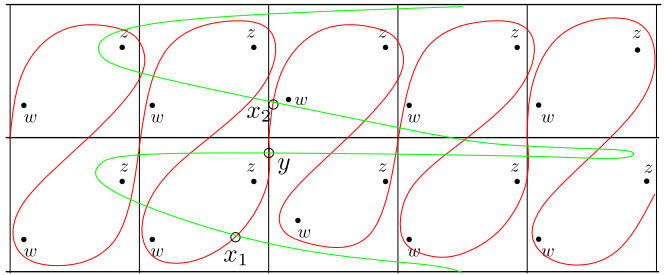}
\caption{Lower type (a), upper type (3).}\label{figure, a+3 sigma negative}
\end{center}
\end{figure}

Second, if the lower row is of type (b) in Figure \ref{figure, lower three endings sigma negative}, then the upper row must be of type (1) in Figure \ref{figure, top three endings sigma negative}, and they combine to generate a pairing diagram of the form as shown in Figure \ref{figure, b+1 sigma negative}. Note in this case $0\leq w(P) <-\sigma-2$. Again, $\tau=\max(A(x_2),A(x_1))$. Note we have $A(x_2)=0$ and $A(x_1)=\frac{w(P)+\sigma}{2}+1\leq 0$. Therefore, $\tau=0$.
\begin{figure}[htb!]
\begin{center}
\includegraphics[scale=0.5]{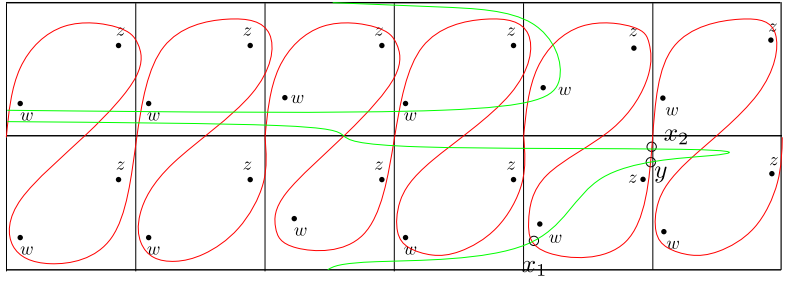}
\caption{Lower type (b), upper type (1).}\label{figure, b+1 sigma negative}
\end{center}
\end{figure}

Finally, if the lower row is of type (c) in Figure \ref{figure, lower three endings sigma negative}, then the upper row must be of type (1) in Figure \ref{figure, top three endings sigma negative}, and the corresponding pairing diagram is shown in Figure \ref{figure, c+1 sigma negative}. Note in this case, $w(P) \geq -\sigma-2$ and $\tau=A(x_0)=\tau=\frac{w(P)+\sigma}{2}+1$.  
\begin{figure}[htb!]
\begin{center}
\includegraphics[scale=0.5]{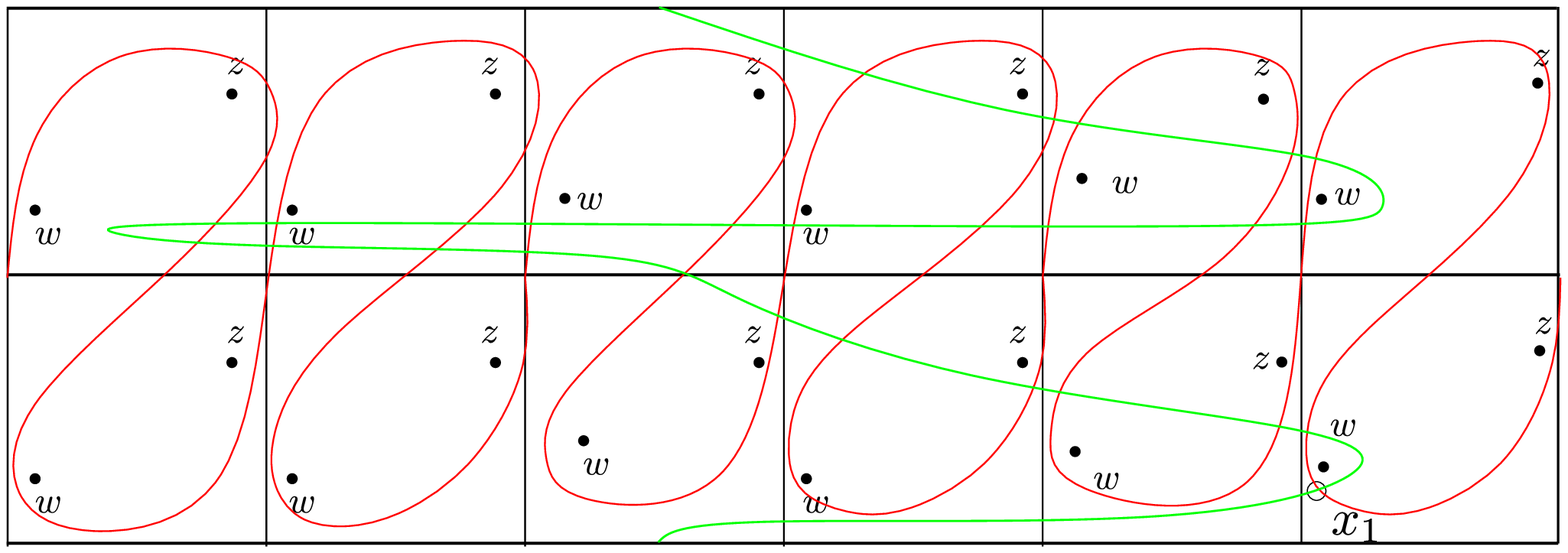}
\caption{Lower type (c), upper type (1).}\label{figure, c+1 sigma negative}
\end{center}
\end{figure}
\end{proof}

\begin{proof}[Proof of Corollary \ref{Corollary, when two-bridge patterns do not homomorphism}]
First note that if a pattern knot $P$ of winding number $w(P)$ induces a homomorphism on the smooth knot concordance group, then $\tau(P(K))=|w(P)|\tau(K)$. To see this, first note $P$ must be a slice pattern to start with, and then a theorem of Roberts in \cite{MR2916283} states there is an number $\epsilon(P)\geq 0$ such that 
$|\tau(P(K))-|w(P)|\tau(K)|\leq \epsilon(P)$ for any companion knot $K$. Suppose $\tau(P(K))-|w(P)|\tau(K)\neq 0$, then one can choose $n$ sufficiently large so that $|\tau(P(nK))-|w(P)|\tau(nK)|=n|\tau(P(K))-|w(P)|\tau(K)|>\epsilon(P)$, which is a contradiction. 

Now by Theorem \ref{Theorem, tau invariant formula}, we may set 
\begin{align*}
\max(\frac{|w(P)|+\sigma}{2}+1,0)=|w(P)|\\
 \min(\frac{-|w(P)|+\sigma}{2},0)=-|w(P)|
\end{align*}
A simple computation implies both equations hold if and only if $|w(P)|=1$ and $\sigma=-1$.
\end{proof}
\section{Brief discussion on immersed curve for general patterns}
We give some speculations on how to extend Theorem \ref{Pairing thm} to involve arbitray pattern knots. 

Without a natural genus-one Heegaard diagram, one has to give a procedure to represent filtered type D structures by immersed train tracks on $T^2$. The difficulty is incurred by the fact that filtered type D structures are often not reduced. The strategy given in \cite{hanselman2016bordered} does not address the unreduced case. In fact, the redundance of differentials in a type D structure causes two issues. First, not every differetial needs to be represented by short arcs in the cutted torus $[0,1]\times [0,1]$, and hence one needs to decide which differentials needed to be drawn. Second, for a differential labeled by $\rho_\emptyset$, one also has to decide on which side of the cutted torus should the corresponding cap be placed. 
\begin{figure}[htb!]
\begin{center}
\includegraphics[scale=0.5]{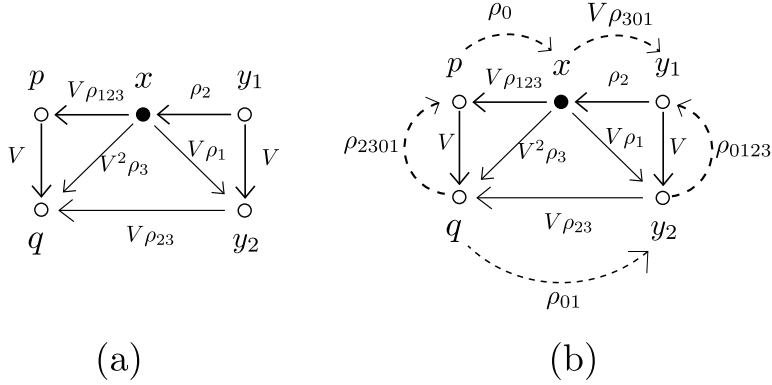}
\caption{(a) A filtered type D structure. (b) A filtered extension of (a). Note $\rho_{\emptyset}$ is omitted in the label as it is the identity of the (extended) torus algebra.}\label{figure, filtered CFD and extension}
\end{center}
\end{figure}
\begin{exam}
We illustrate the issues by an example. Consider the filtered type D structre given in Figure \ref{figure, filtered CFD and extension} (a). Here we view the type D structure as a module over $\mathcal{A}\otimes \mathbb{F}[V]$, where $V$ is a formal variable used to record the shift of the Alexander grading in the structure map. (See Page 203 of \cite{MR3827056}.) If one were to draw the corresponding train track, then for a disirable pairing theorem to hold, one would arrive at Figure \ref{figure, train track for filtered CFD} (a). (Note when pairing this train track with another, we used its elliptic involution as in Figure \ref{figure, train track for filtered CFD} (b) which can be viewed as certain ``immersed Heegaard diagram".) Notice all the edges are drawn using the rules given in Section 2.2, but we throw away the edges corresponding to $x \xrightarrow{V^2\rho_3} q$ and $x \xrightarrow{V\rho_1} y_2$. Also one needs to prevent messing up with the order of $p$ and $q$, i.e.\ arriving at a diagram in Figure \ref{figure, messed-up train track}.
\end{exam}
\begin{figure}[htb!]
\begin{center}
\includegraphics[scale=0.45]{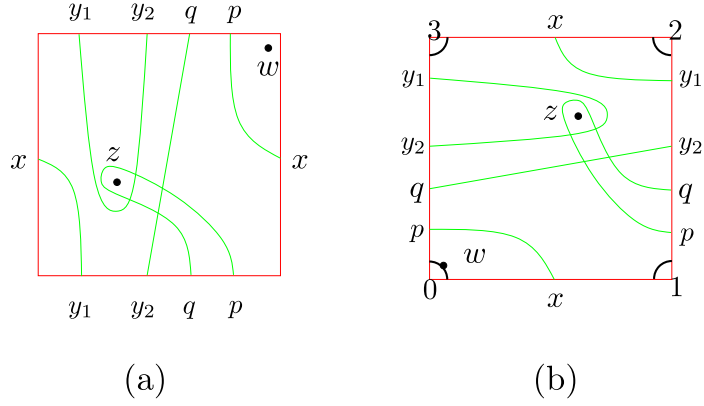}
\caption{(a) The train track corresponding to the filtered type D structure in Figure \ref{figure, filtered CFD and extension} (a). (b) The elliptic involution of (a), which is used when doing Lagrangian intersection pairing; this may be viewed as an ``immersed Heegaard diagram" and used to compute the type D structure. }\label{figure, train track for filtered CFD}
\end{center}
\end{figure}

\begin{figure}[htb!]
\begin{center}
\includegraphics[scale=0.5]{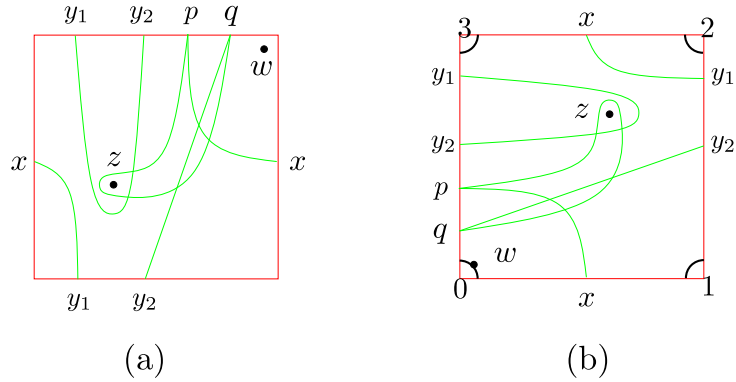}
\caption{(a) A bad train track representation for the type D structure in Figure \ref{figure, filtered CFD and extension} (a) due to poor position of $p$ and $q$. (b) The elliptic involution of (a).}\label{figure, messed-up train track}
\end{center}
\end{figure}
Such issues can be resolved by introducing a notion called \emph{filtered extendability}. To spell out, recall the torus algebra $\mathcal{A}$ can be extended to a larger algebra $\tilde{\mathcal{A}}$ as shown by the quiver diagram in Figure \ref{figure, extended torus algebra}. We use $\tilde{\mu}$ to denote the multiplication, use $\mathcal{I}$ to denote the ring of idempotents, and let $U$ denote the central element $\rho_{0123}+\rho_{1230}+\rho_{2301}+\rho_{3012}$.

\begin{figure}[htb!]
\begin{center}
\includegraphics[scale=0.5]{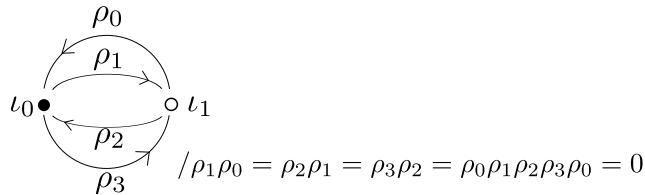}
\caption{The extended torus algebra}\label{figure, extended torus algebra}
\end{center}
\end{figure}
\begin{defn}
A filtered extended type D structure over $\tilde{\mathcal{A}}$ is is a unital left $\mathcal{I}\otimes \mathbb{F}[V]$-module $N$ equipped with an $\mathcal{I}\otimes \mathbb{F}[V]$-linear map $\tilde{\delta}: N \rightarrow \tilde{\mathcal{A}}\otimes_{\mathcal{I}} N$ satisfying the compatibility condition $$(\tilde{\mu}\otimes \mathbb{I})\circ(\mathbb{I} \otimes \tilde{\delta} )\circ \tilde{\delta}(x)=UV\otimes x.$$
\end{defn}

We point out such condition is satisfied automatically for type D structures associated to pattern knots. 

\begin{thm}
Every filtered type D structure arose from some doubly-pointed bordered Heegaard diagram is filtered extendable.
\end{thm} 
\begin{proof}
Literally the same as Appendix A in \cite{hanselman2016bordered}, with an extra base point taken into account. 
\end{proof}

Given a filtered extended type D structure, we can associate a train track to it via the following procedure:
\begin{itemize}
\item[(Step 1)]Reduce the type D structure until it is \emph{filtered reduced}, i.e.\ there is a basis for $N$ so that over this base, no differential is labeled by $\rho_{\emptyset}$ (but could be labeled by $V^n\rho_{\emptyset}$ for some positive integer $n$).
\item[(Step 2)]Represent the filtered reduced structure by a decorated graph and throw away all edges corresponding to differentials with nonzero $V$ power. 
\item[(Step 3)]Embed the vertices of the decorated graph into $T$ so that the $\bullet$ vertices lie on $\alpha$ in the interval $0\times [\frac{1}{4},\frac{3}{4}]$, and the $\circ$ vertices lie on $\beta$ in the interval $[\frac{1}{4},\frac{3}{4}]\times 0$.
\item[(Step 4)]Embed the edges of the decorated graph into $T$ according to rule as shown in Figure \ref{figure,rule of train track, filtered version}. Arrange all the edges to intersect transversely. 
\end{itemize}

\begin{figure}[htb!]
\begin{center}
\includegraphics[scale=0.55]{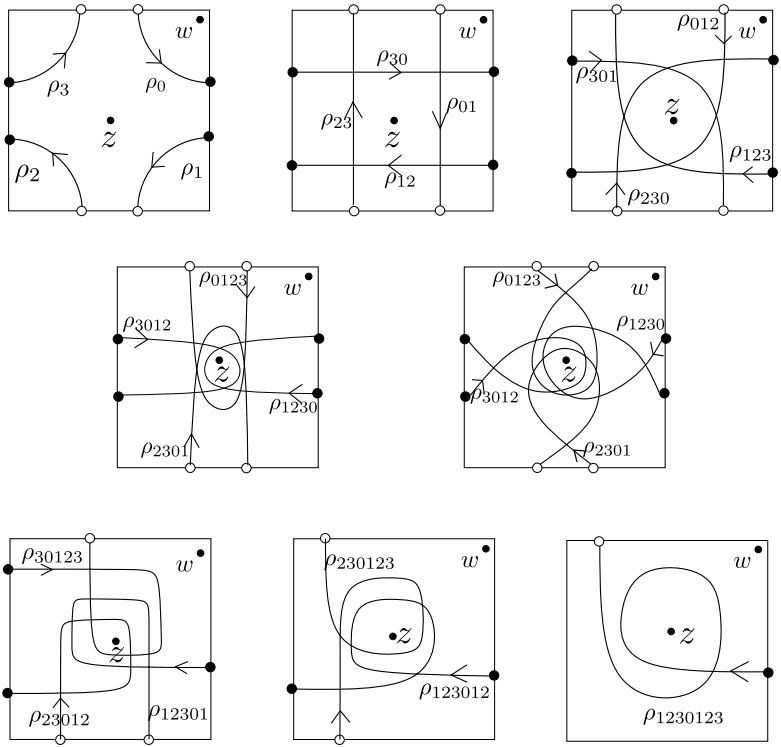}
\caption{Rules for assigning train tracks for a filtered extended type D structure.}\label{figure,rule of train track, filtered version}
\end{center}
\end{figure}

In practice, the immersed train tracks obtained above always form curves, and Lagrangian intersection pairing with such curves recovers the box-tensor product. We illustrate the procedure by examples. 

\begin{exam}
The filtered extended type D structre shown in Figure \ref{figure, filtered CFD and extension} (b) is an extension of the type D structure considered in Figure \ref{figure, filtered CFD and extension} (a). The corresponding train track are given in Figure \ref{figure, train tracks from filtered extension}: we give two different diagrams corresponding to different ordering of $p$ and $q$, but curves in the diagrams are obviously regularly homotopic to each other.  
\end{exam}

\begin{figure}[htb!]
\begin{center}
\includegraphics[scale=0.5]{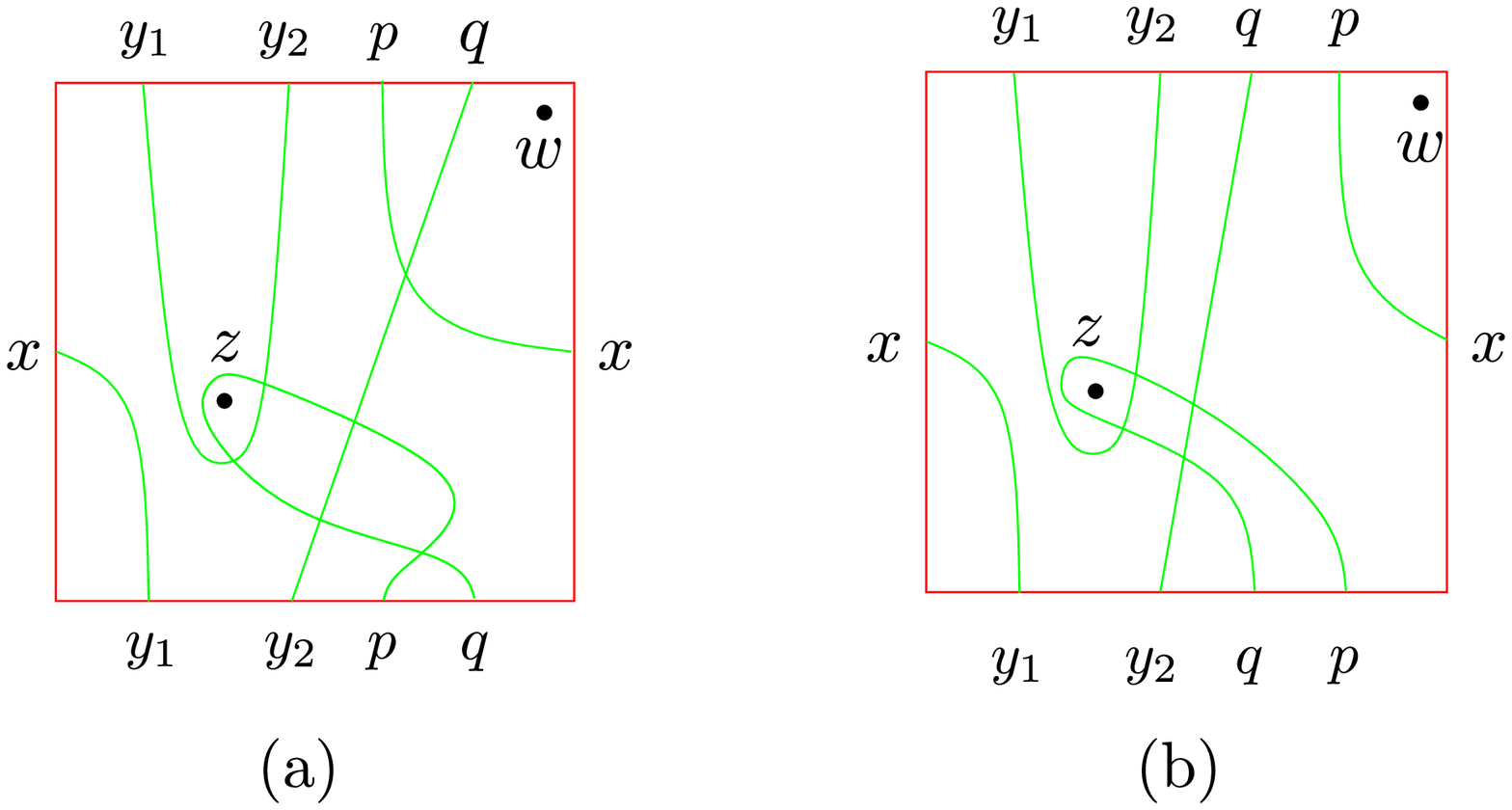}
\caption{Train tracks for the filtered extended type D structure in Figure \ref{figure, filtered CFD and extension}}\label{figure, train tracks from filtered extension}
\end{center}
\end{figure}

\begin{exam}
Note also that all the filtered type D structures coming from a genus-one doubly-pointed bordered Heegaard diagram can be extended to a filtered extended type D structure, and one can check that if one were to represent such type D structures as train tracks by the above algorithm, then one recovers the Heegaard diagram.
\end{exam}

In general it is easy to see that paring such train tracks with immersed curves of knot complements would give the $\widehat{HFK}$-group, as the train tracks thus obtained correspond to associated graded objects of the filtered type D structures. It is not clear to the author, though expected, that such train tracks can be represented as immersed curves. If so, a further question would be if the hat-version filtered knot Floer chain complex can be recovered. The author hope to address these questions in a future project.

\bibliographystyle{abbrv}
\bibliography{satellite}

\begin{thebibliography}{10}

\bibitem{chen2017alexander}
W.~Chen.
\newblock On the {A}lexander polynomial and the signature invariant of
  two-bridge knots.
\newblock {\em arXiv preprint arXiv:1712.04993}, 2017.

\bibitem{hanselman2016bordered}
J.~Hanselman, J.~Rasmussen, and L.~Watson.
\newblock Bordered floer homology for manifolds with torus boundary via
  immersed curves.
\newblock {\em arXiv preprint arXiv:1604.03466}, 2016.

\bibitem{hanselman2018heegaard}
J.~Hanselman, J.~Rasmussen, and L.~Watson.
\newblock Heegaard floer homology for manifolds with torus boundary: properties
  and examples.
\newblock {\em arXiv preprint arXiv:1810.10355}, 2018.

\bibitem{hanselman2019cabling}
J.~Hanselman and L.~Watson.
\newblock Cabling in terms of immersed curves.
\newblock {\em arXiv preprint arXiv:1908.04397}, 2019.

\bibitem{MR2372849}
M.~Hedden.
\newblock Knot {F}loer homology of {W}hitehead doubles.
\newblock {\em Geom. Topol.}, 11:2277--2338, 2007.

\bibitem{MR2511910}
M.~Hedden.
\newblock On knot {F}loer homology and cabling. {II}.
\newblock {\em Int. Math. Res. Not. IMRN}, 12:2248--2274, 2009.

\bibitem{hedden2018satellites}
M.~Hedden and J.~Pinzon-Caicedo.
\newblock Satellites of infinite rank in the smooth concordance group.
\newblock {\em arXiv preprint arXiv:1809.04186}, 2018.

\bibitem{MR3217622}
J.~Hom.
\newblock Bordered {H}eegaard {F}loer homology and the tau-invariant of cable
  knots.
\newblock {\em J. Topol.}, 7(2):287--326, 2014.

\bibitem{MR3589337}
A.~S. Levine.
\newblock Nonsurjective satellite operators and piecewise-linear concordance.
\newblock {\em Forum Math. Sigma}, 4:e34, 47, 2016.

\bibitem{MR3827056}
R.~Lipshitz, P.~S. Ozsvath, and D.~P. Thurston.
\newblock Bordered {H}eegaard {F}loer homology.
\newblock {\em Mem. Amer. Math. Soc.}, 254(1216):viii+279, 2018.

\bibitem{miller2019homomorphism}
A.~N. Miller.
\newblock Homomorphism obstructions for satellite maps.
\newblock {\em arXiv preprint arXiv:1910.03461}, 2019.

\bibitem{MR2708610}
P.~J.~P. Ording.
\newblock {\em On knot {F}loer homology of satellite (1,1) knots}.
\newblock ProQuest LLC, Ann Arbor, MI, 2006.
\newblock Thesis (Ph.D.)--Columbia University.

\bibitem{MR2023281}
P.~Ozsv\'{a}th and Z.~Szab\'{o}.
\newblock Holomorphic disks and genus bounds.
\newblock {\em Geom. Topol.}, 8:311--334, 2004.

\bibitem{MR2065507}
P.~Ozsv\'{a}th and Z.~Szab\'{o}.
\newblock Holomorphic disks and knot invariants.
\newblock {\em Adv. Math.}, 186(1):58--116, 2004.

\bibitem{OS04}
P.~Ozsv\'{a}th and Z.~Szab\'{o}.
\newblock Holomorphic disks and topological invariants for closed
  three-manifolds.
\newblock {\em Ann. of Math. (2)}, 159(3):1027--1158, 2004.

\bibitem{MR3134023}
I.~Petkova.
\newblock Cables of thin knots and bordered {H}eegaard {F}loer homology.
\newblock {\em Quantum Topol.}, 4(4):377--409, 2013.

\bibitem{MR2704683}
J.~A. Rasmussen.
\newblock {\em Floer homology and knot complements}.
\newblock ProQuest LLC, Ann Arbor, MI, 2003.
\newblock Thesis (Ph.D.)--Harvard University.

\bibitem{MR2916283}
L.~P. Roberts.
\newblock Some bounds for the knot {F}loer {$\tau$}-invariant of satellite
  knots.
\newblock {\em Algebr. Geom. Topol.}, 12(1):449--467, 2012.

\end{thebibliography}
\end{document}